\documentclass{scrartcl}
\usepackage[top=20truemm,bottom=30truemm,left=30truemm,right=30truemm]{geometry}
\usepackage{amssymb}
\usepackage{amsmath}
\usepackage{amsthm}
\usepackage{mathrsfs}
\usepackage{amsfonts}
\usepackage{graphicx}
\usepackage{bm}
\usepackage{color}
\usepackage{subcaption}
\usepackage{url}
\usepackage{here}
\definecolor{refkey}{cmyk}{0.64,0.00,0.95,0.40}
\definecolor{labelkey}{rgb}{0.9451,0.2706,0.4941}
\definecolor{mygreen}{cmyk}{0.64,0.00,0.95,0.40}

\definecolor{cyan20}{cmyk}{.2,0,0,0}

\makeatletter
\DeclareOldFontCommand{\rm}{\normalfont\rmfamily}{\mathrm}
\makeatother

\newcommand{\ep}{\varepsilon}
\newcommand{\dual}[1]{\langle{#1}\rangle}

\title{
The inf-sup condition and error estimates of the Nitsche method for
evolutionary diffusion-advection-reaction equations
}
\author{
Yuki Ueda\thanks{Graduate School of Mathematical Sciences, The University of
Tokyo, Komaba 3-8-1, Meguro, Tokyo 153-8914, Japan. \textit{E-mail}: \texttt{ueda@ms.u-tokyo.ac.jp}}
 \ and 
Norikazu Saito\thanks{Graduate School of Mathematical Sciences, The University of
Tokyo, Komaba 3-8-1, Meguro, Tokyo 153-8914, Japan. \textit{E-mail}: \texttt{norikazu@ms.u-tokyo.ac.jp}}
}
\begin{document}

%%% THEOREM STYLES
\theoremstyle{plain}
\newtheorem{theorem}{Theorem}%[section]
\newtheorem{lemma}{Lemma}[section]
\theoremstyle{definition}
\newtheorem{definition}[lemma]{Definition}%[section]
\newtheorem{example}[lemma]{Example}
\newtheorem{remark}[lemma]{Remark}
\newtheorem{assumption}{Assumption}
\renewcommand\theassumption{\Roman{assumption}}

%%%%%%%%%%%%%%%
%%%%%%%%%%%%%%%
%%%%%%%%%%%%%%%
\maketitle

%%%%%%%%%%%%%%%
%%%%%%%%%%%%%%%
\begin{abstract}
The Nitsche method is a method of ``weak imposition'' of the
 inhomogeneous Dirichlet boundary conditions for partial differential
 equations. This paper explains stability and convergence study of the
 Nitsche method applied to evolutionary diffusion-advection-reaction
 equations. We mainly discuss a general space semidiscrete scheme 
 including not only the standard finite element method but also
 Isogeometric Analysis. Our method of analysis is a variational one
 that is a popular method for studying elliptic problems. The
 variational method enables us to obtain the best approximation property directly.
 Actually, results show that the scheme satisfies the inf-sup condition and
 Galerkin orthogonality. Consequently, the optimal order error estimates
 in some appropriate norms are proven under some regularity assumptions
 on the exact solution. We also consider a fully discretized scheme using
 the backward Euler method. Numerical example demonstrate the validity of
 those theoretical results.
\end{abstract}

%%
%% Key words
%%
%\bigskip

\bigskip

{\noindent \textbf{2010 Mathematics Subject Classification:}
65M12, 65M60

%%%%%%%%%%%%%%%%%%%%%%%%%%%%%%%%%%%%%%%%%
%%%%%%%%%%%%%%%%%%%%%%%%%%%%%%%%%%%%%%%%%
\section{Introduction}
\label{sec:1}

The boundary condition is an indispensable component of the well-posed
problem of partial differential equations. It is not merely a side
condition. In computational mechanics, great attention should be paid the imposition of boundary
conditions, although it is sometimes
understood as a simple and unambiguous task.

The Neumann boundary condition or natural boundary condition 
is naturally considered in the variational equation so that it is
handled directly in finite element method (FEM). 
By contrast, a specification of the Dirichlet boundary condition
(DBC) has room for discussion. In traditional FEM including the
continuous $\mathbb{P}^k$ FEM for example, DBC is imposed
by specifying the nodal values at boundary nodal points. Although it is
simple, this ``strong imposition'' of DBC is based on the fact that unknown values of the resulting
finite dimensional system agree with
nodal values in traditional FEM. Therefore, it is difficult to apply
this technique to the iso-geometric analysis
(IGA). Actually, IGA is a class of FEM using B-spline or non-uniform rational
B-spline (NURBS) basis functions. It has been widely
applied in many fields of computational mechanics, providing smooth
approximate solutions of partial differential equations using only a few
degrees of freedom (DOF). Moreover, it provides a more accurate
representation of computational domains with complex shapes. That is,
the geometric representation of 3D computational domain generated by a
CAD system is handled directly. See \cite{chb09} for more details.
Unfortunately, unknown values in IGA do not generally agree with nodal values.
Furthermore, it has often been pointed out that strongly
imposed DBCs give rise to spurious oscillations, even for stable
discretzation methods. To resolve those shortcomings, Bazilevs et
al. \cite{bh07,bmch07} proposed a method of ``weak imposition'' of DBC by
applying the methodology of the discontinuous Galerkin (DG) method and
discussed its efficiency by numerical experiments for non-stationary
Navier--Stokes equations. Their method, originally proposed by
Nitsche \cite{nit71}, is commonly called the \emph{Nitsche method}.       
Stability and convergence of the Nitsche method for
elliptic problems have been well studied to date. 

This paper addresses the Nitsche method for evolutionary problems. In
particular, we study the stability and convergence of the FEM
discretization including IGA. Earlier studies of the Nitsche method were
conducted by formulating the method as a one-step method, as in an earlier report of the literature 
\cite{tho06}. By
contrast, we present a different perspective: we
assess the Nitsche method using a variational approach. 
Consequently, the analysis becomes greatly simplified. Optimal
order error estimates in some appropriate norms are established.
Such a variational approach was recently applied successfully to analysis
of the DG time-stepping method for a wide class of parabolic equations
in an earlier paper \cite{sai17}.   
To fix the idea, we consider the following
diffusion-advection-reaction equation for the function $u=u(x,t)$,
$x\in\overline\Omega$ and $t\in \overline{J}$, 
\begin{subequations}
 \label{eq:1}
  \begin{align}
\partial_t u+L(x,t)u&=f(x,t), && (x,t)\in\Omega\times J, \label{eq:1a}\\
u(x,t)&=g(x,t) ,&& (x,t)\in\Gamma\times J, \label{eq:1b}\\
   u(x,0)&=u_0(x) ,&& x\in\Omega .\label{eq:1c}
\end{align}
Hereinafter, $\Omega$ is a bounded polyhedral domain or NURBS domain (see Definition \ref{Def:NURBS_Domain}) in
$\mathbb{R}^d$ with the boundary $\Gamma
=\partial\Omega$, $J$ represents the time interval defined as $J=(0,T)$ with $T>0$, and  
$L(x,t)$ signfies the elliptic differential operator defined as 
\begin{equation}
 \label{eq:Lw}
 L(x,t)w=-\nabla \cdot \mu(x,t)\nabla w+{b}(x,t)\cdot \nabla w + c(x,t)w.
\end{equation}
\end{subequations}
Moreover,
$\mu:\Omega\times J\to\mathbb{R}^{d\times d}$, 
${b}:\Omega\times J\to\mathbb{R}^{d}$, 
$c,f:\Omega\times J\to\mathbb{R}$, 
$g:\Gamma\times J\to\mathbb{R}$, 
and $u_0:\Omega\to\mathbb{R}$ are given functions. The
assumptions to these functions will be described later.

At this stage, we describe the idea of applying the Nitsche method for
\eqref{eq:1} to clarify the novelty and motivation of this study. To avoid unimportant difficulties, we presume that
$\mu=I$, $b=0$ and $c=0$ for the time being. In subsequent sections, we remove those
restrictions.    
By multiplying both sides of \eqref{eq:1a} by a test function $v\in
H^1(\Omega)$, integrating over $\Omega$ and finally applying integration by
parts, we obtain
%\begin{multline*}
% \label{eq:forma}
\[
 \int_\Omega (\partial_tu) v~dx +\int_\Omega\nabla u\cdot\nabla
 v~dx-\int_{\Gamma} (n\cdot \nabla u)v~dS=\int_\Omega fv~dx.
\]
 %\end{multline*}
%where
%\begin{equation}
% \label{eq:forma}
% a(t;w,v)=\int_\Omega \left[\mu\nabla w\cdot \nabla v
%		       +({b}\cdot \nabla w)v + cwv
%		       \right]~ dx.
%\end{equation}
Introducing a partition $\mathcal{T}_h$ of $\Omega$, with $h$ being the granularity parameter, and a finite
dimensional subspace $V_h$ of $H^1(\Omega)$, we consider the Galerkin approximation
\begin{equation*}
%\label{eq:formal1}
 \int_\Omega (\partial_tu_h) v_h~dx +\int_\Omega\nabla u_h\cdot\nabla
 v_h~dx -\sum_{E\in\mathcal{E}_h^e}\int_{E} (n\cdot \nabla u_h)v_h~dS\\
 =\int_\Omega fv_h~dx\quad
 (v_h\in V_h),
\end{equation*}
where $\mathcal{E}_h^e$ denotes the set of all boundary edges.
(the definition of those notations will be stated in Section \ref{sec:2}.) 
Then, the Nitsche method reads as shown below
\begin{multline*}
 \int_\Omega (\partial_tu_h) v_h~dx +\int_\Omega\nabla u_h\cdot\nabla
 v_h~dx
 -\sum_{E\in\mathcal{E}_h^e}\int_{E} (n\cdot \nabla u_h)v_h~dS\\
 \underbrace{-\sum_{E\in\mathcal{E}_h^e}\int_{E} (n\cdot \nabla v_h)(u_h-g)~dS}_{=I_s}
 \underbrace{+\sum_{E\in\mathcal{E}_h^e}\int_{E} \frac{\ep}{h}(u_h-g)v_h~dS}_{=I_p}
 =
\int_\Omega fv_h~dx\quad
 (v_h\in V_h)
\end{multline*}
for a.e. $t\in J$, where $\ep>0$ is a penalty parameter. This is written, equivalently, as
\begin{equation}
\label{eq:formal2}
\displaystyle\int_{\Omega}(\partial_tu_h)v_h~dx+a_{\ep,h}(u_h,v_h)=F_{\ep,h}(t;v_h)\qquad (v_h\in V_h,~t\in J), 
\end{equation}
where 
\begin{align*}
a_{\ep,h}(u_h,v_h)&= %\int_\Omega (\partial_tu_h) v_h~dx +
\int_\Omega\nabla u_h\cdot\nabla
 v_h~dx 
 -\sum_{E\in\mathcal{E}_h^e}\int_{E} (n\cdot \nabla u_h)v_h~dS\\
 &-\sum_{E\in\mathcal{E}_h^e}\int_{E} (n\cdot \nabla v_h)u_h~dS +\sum_{E\in\mathcal{E}_h^e}\int_{E} \frac{\ep}{h}u_hv_h~dS\\
F_{\ep,h}(t;v_h)&= \int_\Omega fv_h~dx
 -\sum_{E\in\mathcal{E}_h^e}\int_{E} (n\cdot \nabla v_h)g~dS+\sum_{E\in\mathcal{E}_h^e}\int_{E} \frac{\ep}{h} gv_h~dS.
\end{align*}
Term $I_s$ is added to symmetrize the equation. Term $I_p$ is
called the penalty term. Letting $\ep$ be sufficiently large, we expect
that the boundary condition $u_h=g$ on $\Gamma$ is specified in a weak
sense. An important advantage of \eqref{eq:formal2} is that the ``elliptic part'' 
\[
 \int_\Omega\nabla u_h\cdot\nabla
 v_h~dx 
 -\sum_{E\in\mathcal{E}_h^e}\int_{E} (n\cdot \nabla u_h)v_h~dS
 -\sum_{E\in\mathcal{E}_h^e}\int_{E} (n\cdot \nabla v_h)u_h~dS +\sum_{E\in\mathcal{E}_h^e}\int_{E} \frac{\ep}{h}u_hv_h~dS
\]
can be coercive in an appropriate norm by choosing suitably
large $\ep$. Moreover, the constant appearing in the coercive inequality is
independent of the penalty parameter $\ep$, which implies that the scheme
can be stable in a certain sense. In fact, the classical penalty
method has no such property. Another advantage is that the smooth solution $u$ of
\eqref{eq:1} exactly satisfies \eqref{eq:formal2}. Consequently, the
``parabolic Galerkin orthogonality''
\begin{equation}
\label{eq:formal3}
\int_J \left[ \int_{\Omega}(\partial_tu-\partial_tu_h)y_h~dx+a_{\ep,h}(u-u_h,y_h)\right]dt=0 \qquad (y_h\in L^2(J;V_h)) 
\end{equation}
is available. This characteristic enables us to apply the variational method to study
the Nitsche method \eqref{eq:formal2}: after having established
the ``inf-sup'' condition, we can derive best
approximation properties and optimal order error estimates directly by combining
the ``inf-sup'' condition and \eqref{eq:formal3}.
Therefore, our effort will be concentrated on the derivation of the
``inf-sup'' condition, which is the main result of this paper.
Although such an approach is quite standard for elliptic problems,
apparently little has been done for parabolic problems.
The use of \eqref{eq:formal3} is not originally our idea. 
%The identity \eqref{eq:formal3} itself is known among specialists 
Others have considered this identity before, but no report of the relevant literature describes systematic use of \eqref{eq:formal3}. 

Before concluding this Introduction, we review earlier studies of the convergence of the Nitsche method applied to parabolic equations. 
Thom{\' e}e \cite{tho06} reported error estimates in the $L^{\infty}(J;L^2(\Omega))$ norm for a semi-discrete (in space) finite  element approximation to a linear inhomogeneous heat equation. 
Heinrich and Jung \cite{hj08} applied the method to a parabolic interface problem, deriving similar error estimates as \cite{tho06}. 
Choudury and Lasiecka \cite{cl91} studied a parabolic diffusion-reaction problem and proved error estimates in the $L^p(J;L^2(\Omega))$ norms with $1\le p\le \infty$ using the semigroup theory. 
All those studies relied on the assumption that the coefficients of the equation are independent of the time variable. 
By contrast, we study the parabolic diffusion-advection-reaction equation with time-dependent coefficients and derive error estimates in the $H^1(J;H^{-1}(\Omega))$ and $L^2(J;H^1(\Omega))$ norms.

This paper is organized as follows. Section \ref{sec:2} states the formulation of Nitsche method and our main results. Section \ref{sec:3} presents a review of some properties of classical FEM and IGA. Section \ref{sec:4}, \ref{sec:5} and \ref{sec:6} provide proof of our main results. Analysis of the fully discretized problem is presented in Section \ref{sec:7}. Finally, this report presents a numerical example in Section \ref{sec:8}.

\section{Nitsche method and the main results}
\label{sec:2}

\subsection{Weak formulation of \eqref{eq:1}}
\label{sec:2z}

We use the standard Lebesgue spaces $L^p(\Omega)$, $L^p(\Gamma)$ and
Sobolev spaces $H^m(\Omega)$, $H^{m-1/2}(\Gamma)$, where $1\le p\le
\infty$ and $1\le m\in \mathbb{Z}$. The norms are denoted as
$\|\cdot\|_{L^p(\Omega)}$,
$\|\cdot\|_{L^p(\Gamma)}$,
$\|\cdot\|_{H^m(\Omega)}$ and 
$\|\cdot\|_{H^{m-1/2}(\Omega)}$ for example. Moreover, the $L^2$-inner product is denoted as $(\cdot,\cdot)_{L^2(\Omega)}$, and so on. The semi-norm
$|\cdot|_{H^m(\Omega)}$ of
$H^m(\Omega)$ is defined as
\[
  |v|^2_{H^m(\Omega)}=\sum_{|\alpha|=m}\left\|\frac{\partial^\alpha}{\partial
   x_1^{\alpha_1}\cdots\partial x_d^{\alpha_d}} v\right\|_{L^2(\Omega)}^2,
\]
where $\alpha=(\alpha_1,\ldots,\alpha_d)$, $0\le \alpha_1,\ldots,\alpha_d\in\mathbb{Z}$,
and $|\alpha|=\alpha_1+\cdots+\alpha_d$. In fact,
\[
 \|v\|_{H^m(\Omega)}^2=\|v\|_{L^2(\Omega)}^2+\sum_{1\le |\alpha|\le m}|v|^2_{H^m(\Omega)}.
\]

Let $\operatorname{Tr}=\operatorname{Tr}(\Omega,{\Gamma})$ be a trace
operator from $H^1(\Omega)$ into $H^{1/2}(\Gamma)$, which is a linear and
 continuous operator. There exists a linear and continuous operator $\operatorname{Tr}^{-1}=\operatorname{Tr}^{-1}(\Omega,\Gamma)$ of
$H^{1/2}(\Gamma)\to H^1(\Omega)$, which is called a lifting operator,
such that $\operatorname{Tr}(\operatorname{Tr}^{-1}\eta)=\eta$ on $\Gamma$ for all $\eta\in H^{1/2}(\Gamma)$. 
Below, we write it as $v|_\Gamma=\operatorname{Tr}v$ if there is no fear of confusion.

As usual, we set $H^1_0(\Omega)=\{v\in H^1(\Omega)\mid v|_\Gamma=0\}$
and
\[
 H^{-1}(\Omega)=[H^1_0(\Omega)]'=\mbox{the dual space of
 $H^1_0(\Omega)$}. 
\]

Let $X$ be a Hilbert space. For $1\leq r < \infty$ and $0\le t_0<t_1$, the space $L^r(t_0,t_1;X)$ denotes a Bochner space equipped with the norm
\begin{equation*}
 \|v\|_{L^r(t_0,t_1;X)} =
  %\left\{\begin{array}{ll}
  \left(\displaystyle\int_{t_0}^{t_1}\|v(t)\|_X^r~dt\right)^{1/r}.
   %& (1\leq r < \infty)\\[3mm]
%\operatorname{esssup}_{t\in (T_0,T_1)}\|v(t)\|_X  & (r=\infty).
%	 \end{array}\right.
\end{equation*}
Let $Y$ be a (possibly another) Hilbert space.  
We also use the so-called Bochner--Sobolev space $H^{1}(t_0,t_1;X,Y)$ defined as
\begin{equation*}
H^{1}(t_0,t_1;X,Y) = \left\{v\in L^2(t_0,t_1;X) \mid \dfrac{\mathrm{d}v}{\mathrm{d}t}\in L^2(t_0,t_1;Y)\right\},
\end{equation*}
where $\frac{\mathrm{d}}{\mathrm{d}t}$ denotes the weak derivative
for $t$. This is a Hilbert space equipped with the norm
\begin{equation*}
\|v\|_{H^1(t_0,t_1;X,Y)}^2 = \|v\|_{L^2(t_0,t_1;X)}^2 + \left\|\dfrac{\mathrm{d}v}{\mathrm{d}t}\right\|_{L^2(t_0,t_1;Y)}^2.
\end{equation*}

%semi-norm
%\begin{equation*}
%|v|_{H^m(t_0,t_1;X,Y)} =
% \left\|\dfrac{\mathrm{d}^mv}{\mathrm{d}t^m}\right\|_{L^2(t_0,t_1;Y)}.
 %= \operatorname{esssup}_{t\in (T_0,T_1)}\left\|\dfrac{\mathrm{d}^iv}{\mathrm{d}t^i}(t)\right\|_X
%\end{equation*}
%and the norm 
%\begin{equation*}
%\|v\|_{W^{k,\infty}(T_0,T_1;X)} =\max_{0\le i\le k} |v|_{W^{i,\infty}(T_0,T_1;X)}. 
%\end{equation*}

It is apparent that $H^1(t_0,t_1;X,X)\subset C^0([t_0,t_1];X)$ (see \cite[theorem2, Chapter 5.9]{eva10} for example) is satisfied. Furthermore, letting $H$ and $\mathscr{V}$ be (real) Hilbert spaces such that $\mathscr{V}\subset H$ is dense with continuous injection, we identify $H$ with $H'$ ($H\simeq H'$) as usual and consider the triple 
\begin{equation}
 \label{eq:gt}
  %\mathscr{V}\hookrightarrow Hs \hookrightarrow \mathscr{V}',
  \mathscr{V}\subset H\subset\mathscr{V}'.
\end{equation}
%where $\mathscr{V}\hookrightarrow H$ means that $\mathscr{V}\subset H$
%is dense with the continuous injection. 
Then we have $H^1(t_0,t_1;\mathscr{V},\mathscr{V}')\subset C^0([t_0,t_1];H)$ (see \cite[Theorem1, Chapter XVIII]{dl92} for example).

\medskip

Throughout this paper, we use the following assumptions: 

\medskip

\begin{assumption}
\label{Assume:Ellipticity}
Regularity of coefficients and data functions:  
%\noindent \textbf{(A0)} Regularity of coefficients and data functions. 
\begin{subequations}
 \label{eq:L}
\begin{gather}
\mu \in C^0(\overline{\Omega\times J})^{d\times d},~ \mu\mbox{ is symmetric};\label{eq:L1}\\
{b} \in L^\infty(J;W^{1,\infty}(\Omega)^d);\label{eq:L2}\\
c \in L^\infty(J;L^\infty(\Omega));\label{eq:L3} \\
\exists \mu_1>\mu_0>0,\quad \mu_0|\xi|^2\le \mu(x,t)\xi\cdot \xi\le
 \mu_1|\xi|^2\quad (x\in\overline{\Omega},t\in J,\xi\in\mathbb{R}^d);\label{eq:L4}\\
 \exists c_0>0,\quad c-\frac{1}{2}\nabla\cdot {b}\ge c_0\quad (x\in\Gamma,t\in J);\label{eq:L5}\\
f\in L^2(J;L^2(\Omega)),\quad u_0\in L^2(\Omega),\quad g\in G_0.\label{eq:L6}
% \bm{b}\cdot n= 0\quad (x\in\Gamma,t\in J).\quad
% \cred{\mbox{$\leftarrow$ artificial assumption!}}\label{eq:L6}
\end{gather}
\end{subequations} 
Therein,
\[
 G_0=\{\eta=w|_\Gamma\in L^2(J;H^{1/2}(\Gamma)) \mid w\in
 L^2(J;H^1(\Omega))\cap H^1(J;L^2(\Omega))\}.
\]
\end{assumption}

Introducing the bilinear form on $H^1(\Omega)\times H^1(\Omega)$ as   
\begin{equation}
 \label{eq:forma}
a(t;w,v)=\int_\Omega \left[\mu\nabla w\cdot \nabla v
		       +({b}\cdot \nabla w)v + cwv \right]~ dx,
\end{equation}
we have
\begin{subequations}
 \label{eq:a1}
\begin{align}
 |a(t;w,v)|&\le M \|w\|_{H^1(\Omega)}\|v\|_{H^1(\Omega)} && (w,v\in H^1(\Omega),t\in J); \label{eq:a11}\\
 a(t;w,w) &\ge  \alpha \|w\|_{H^1(\Omega)}^2 &&  (w\in
 H^1_0(\Omega),t\in J);\label{eq:a12}\\
 a(t;w,v)&=(L(\cdot,t)w,v)_{L^2(\Omega)} && (w\in H^2(\Omega),v\in
 H^1_0(\Omega),t\in J),\label{eq:a13}
\end{align}
\end{subequations}
where $M=M(d,\mu,{b},c)>0$ and $\alpha=\min\{\mu_0,c_0\}$. 

The following is the standard result (see \cite[Chapter XVIII]{dl92}, \cite[Chapter IV]{wlo87} for example).
\begin{lemma}
\label{la:00}
Presuming that Assumption \ref{Assume:Ellipticity} is satisfied, then there exists a unique
 \begin{subequations} 
  \label{eq:u0}
\begin{equation}
\label{eq:u00}
u\in
H^1(J;H^1(\Omega),H^{-1}(\Omega))
\end{equation}
such that   
\begin{equation}
\label{eq:u01}
u=g\quad \mbox{on }\Gamma,~t\in J,
\end{equation}
 and 
\begin{multline}
 \int_0^T [\dual{\partial_t{u},y_1}_{H^{-1},H^1_0}+a(t;u(t),y_1(t))]~dt+({u}(0),y_2)_{L^2(\Omega)} \\ =
  \int_0^T (f,y_1)_{L^2(\Omega)}~dt +(u_0,y_2)_{L^2(\Omega)}\qquad \forall
  (y_1,y_2)\in L^2(J;H^1_0(\Omega))\times L^2(\Omega).
\label{eq:u02}
\end{multline}
 \end{subequations}
Moreover, we have $\partial_tu+L(t)u\in L^2(J;L^2(\Omega))$ and it holds that
 \eqref{eq:1} for $x\in \Omega$ and $t\in J$.  
\end{lemma}
%\end{prop}

% \begin{lemma}
%\label{p:2}
% \end{lemma}

\subsection{Finite dimensional subspaces}
\label{sec:2a}

We introduce a finite dimensional subspace $V_h$ of $H^1(\Omega)$
in a somewhat abstract manner below. Concrete examples are given in
Section \ref{sec:3}. We also collect (finite dimensional) function
spaces and norms used for this study. 

Recall that $\Omega\subset\mathbb{R}^d$ is a polyhedral domain with the
boundary $\Gamma$.
%, then we have the trace operator $\operatorname{Tr}:H^1(\Omega)\to H^{1/2}(\Gamma)$. Hereinafter, $Tu\in H^{1/2}(\Gamma)$ is written as $u$ for $u\in H^1(\Omega)$. 
We introduce a partition $\mathcal{T}_h$ of $\Omega$ such that each
$K\in\mathcal{T}_h$ is a closed set in $\mathbb{R}^d$, 
the $\mathbb{R}^d$ Lebesgue measure of $K\cap K'$ vanishes for $K,K'\in\mathcal{T}_h$ with $K\ne K'$, and
\[
 \overline{\Omega}=\bigcup_{K\in \mathcal{T}_h}K.
\]
The diameter of $K\in\mathcal{T}_h$ is designated by $h_K$ and is set as
$h=\max\{h_K \mid K\in \mathcal{T}_h\}$. Then, letting $\mathcal{E}_h$ be the set of edges and
$\mathcal{E}_h^e =\{E\in \mathcal{E}_h \mid E\subset \Gamma\}$, we
express $\Gamma$ as
\[
		   \Gamma=\bigcup_{E\in\mathcal{E}_h^e}E.
\]
For $E\in\mathcal{E}_h$, the diameter of $E$ is designated by $h_E$. 
Moreover, for $E\in\mathcal{E}_h$, we write $K_E$ to express
$K\in\mathcal{T}_h$ such that $E\subset \partial K$. In general, such
$K_E$ is not unique. However, it is unique for any $E\in\mathcal{E}_h^e$.

\begin{assumption}
\label{Assume:EqualityatExternalElement}
There exists a positive constant $C$ such that
\begin{equation}
 \label{eq:ke}
h_{K_E}\le Ch_E\qquad (E\in\mathcal{E}_h^e,~K_E\in\mathcal{T}_h).
\end{equation}
\end{assumption}

%In the next chapter, we will state the Nitsche's method for parabolic problems using the finite dimensional subspace $V_h\subset H^1(\Omega)$. The following assumptions will be needed throughout the chapter.

Below, we use the finite dimensional subspace
\begin{equation}
 \label{eq:vh00}
V_h\subset H^1(\Omega).
\end{equation}

We mention no specific definition, but we do make the following assumptions. 

\begin{assumption}
\label{Assume:Local_Regularity}
\begin{equation}
v_h|_K\in H^2(K)\qquad(v_h\in V_h,~ K\in\mathcal{T}_h).
\end{equation}
\end{assumption}

\begin{assumption}
\label{Assume:inequalities}
%Let $V_h$ be the finite dimensional subspace of $H^1(\Omega)$ which relates to $\mathcal{T}_h$. Then we assume that (A1)-(A3) hold,
(i) Trace inequality. There exists a positive constant $C_{\mathrm{Tr}}$ such that
\begin{equation}
\|v\|_{L^2(E)}^2 \le C_{\mathrm{Tr}}h_E^{-1}\left(\|v\|_{L^2(K_E)}^2 +
		       h_{K_E}^2|v|_{H^1(K_E)}^2\right)\quad (v\in
H^1(K_E),~ E\in\mathcal{E}_h^e).
\end{equation}
%where $K_E\in\mathcal{T}_h$ is an element such that $E\subset \partial K_E$, for $E\in$.
(ii) Inverse inequality.
\begin{equation}
|v_h|_{H^1(K)}\le Ch_K^{-1}\|v_h\|_{L^2(K)} \quad (v_h\in V_h,
 ~K\in\mathcal{T}_h).
\end{equation}
(iii) Interpolation error estimates. There exists a positive integer $k$
 and a projection $\Pi_h:H^{k}(\Omega)\to V_h$ such that, for $2\le l\le
 k+1$, 
\begin{equation}
\|w-\Pi_hw\|_{H^j(\Omega)}\le Ch^{l-j}\|w\|_{H^{l}(\Omega)}\quad 
(w\in H^{l}(\Omega),~j=0,1,2).
\end{equation}
\end{assumption}

Assumptions \ref{Assume:EqualityatExternalElement} and \ref{Assume:inequalities} imply that there exists a positive constant $C^*$ such that
\begin{equation}
\label{Eq:Trace_Constant}
\|v_h\|_{L^2(E)}^2\le C^*h_E^{-1}\|v_h\|_{L^2(K_E)}^2\quad(v_h\in V_h,~ E\in\mathcal{E}_h^e).
\end{equation}
Moreover, Assumption \ref{Assume:Local_Regularity} gives that the same constant $C^*$ satisfies
\begin{equation}
\label{Eq:Trace_Constant2}
\|n\cdot\nabla v_h\|_{L^2(E)}^2 \le C^*h_E^{-1}|v_h|_{H^1(K_E)}^2\quad(v_h\in V_h,~ E\in\mathcal{E}_h^e).
\end{equation}

Setting $V = \{v\in H^1(\Omega)\mid v|_K\in H^2(K)\quad (K\in\mathcal{T}_h)\}$, then $V_h\subset V$. Furthermore, we define
\begin{align*}
\|v_h\|_{V_h}^2 &= \|v_h\|_{H^1(\Omega)}^2 + \displaystyle\sum_{E\in\mathcal{E}_h^e}h_E^{-1}\|v_h\|_{L^2(E)}^2,\\
\|v\|_{V}^2 &= \|v\|_{H^1(\Omega)}^2 + \displaystyle\sum_{K\in\mathcal{T}_h}h_K^2|v|_{H^2(K)}^2 + \sum_{E\in\mathcal{E}_h^e}h_E^{-1}\|v\|_{L^2(E)}^2.
\end{align*}
This definition implies that $\|v_h\|_{V_h}\le \|v_h\|_V\le C\|v_h\|_{V_h}$ for all $v_h\in V_h$, where $C$ is a positive constant.
Moreover, for $\phi\in L^2(\Omega)$, we write that
\begin{equation}
\|\phi\|_{V_h'} = \displaystyle\sup_{v_h\in
 V_h}\dfrac{(\phi,v_h)_{L^2(\Omega)}}{\|v_h\|_{V_h}}.
\end{equation}
It is apparent that $\|\phi\|_{V_h'}\le \|\phi\|_{L^2(\Omega)}$ for every $\phi\in L^2(\Omega)$.

Furthermore, we define the space of trial function and test function in the Nitsche method. Let
\begin{equation}
X_h = H^1(J;V_h,V_h), \quad Y_h = L^2(J;V_h)\times V_h.
\end{equation}
They are Hilbert spaces equipped with the norms
\begin{align*}
\|z_h\|_{X_h}^2 &= \displaystyle\int_J(\|z_h\|_{V_h}^2 + \|\partial_tz_h\|_{V_h'}^2)~dt + \|z_h(0)\|_{L^2(\Omega)}^2,\\
\|(y_h,\widetilde{y}_h)\|_{Y_h}^2 &= \displaystyle\int_J\|y_h\|_{V_h}^2~dt +\|\widetilde{y}_h\|_{L^2(\Omega)}^2,
\end{align*}
respectively. In fact, $X_h\subset C^0(\overline{J};V_h)$. We also define the space
\begin{equation*}
X_V = \{z\in H^1(J;H^1(\Omega),L^2(\Omega))\mid z(t)\in V\quad (t\in J)\},
\end{equation*}
and norm
\begin{equation*}
\|z\|_{X_V}^2 = \displaystyle\int_J (\|z\|_{V}^2 + \|\partial_tz\|_{V_h'}^2)~dt+\|z(0)\|_{L^2(\Omega)}^2,
\end{equation*}
which satisfies $X_h\subset X_V\subset C^0(\overline{J};L^2(\Omega))$ and $\|z_h\|_{X_h}\le \|z_h\|_{X_V}\le C\|z_h\|_{X_h}$ for all $z_h\in X_h$, where $C$ is a positive constant.

\subsection{Formulation of the Nitsche method}
\label{sec:2b}

The Nitsche method for parabolic problems is presented below.

\smallskip

\noindent \textbf{(P$_{\ep,h}$)} Find $u_{\ep,h}\in X_h$ such that
\begin{subequations}
\label{Eq:Nitsche_for_each_time}
\begin{align}
\int_{\Omega}(\partial_tu_{\ep,h})v_h~dx + a_{\ep,h}(t;u_{\ep,h}(t),v_h) &= F_{\ep,h}(t;v_h)&&(v_h\in V_h,~t\in J),\label{Eq:Nitsche_for_each_time1}\\
\int_{\Omega}u_{\ep,h}(0)\widetilde{v}_h~dx &=
 \int_{\Omega}u_0\widetilde{v}_h~dx&& (\widetilde{v}_h\in V_h) \label{Eq:Nitsche_for_each_time2}.
\end{align}
\end{subequations}
Therein, we set
\begin{align*}
a_{\ep,h}(t;w,v_h)&=  a(t;w,v_h)
 -\sum_{E\in\mathcal{E}_h^e}\int_{E} (n\cdot\mu\nabla w)v_h~dS\\
 &\quad-\sum_{E\in\mathcal{E}_h^e}\int_{E} (n\cdot\mu\nabla v_h)w~dS - \int_{\Gamma_{\mathrm{in}}}b\cdot nwv_h~dS +\sum_{E\in\mathcal{E}_h^e}\int_{E} \frac{\ep}{h_E}wv_h~dS\\
F_{\ep,h}(t;v_h)&= \int_\Omega fv_h~dx
 -\sum_{E\in\mathcal{E}_h^e}\int_{E} (n\cdot\mu \nabla v_h)g~dS - \int_{\Gamma_{\mathrm{in}}}b\cdot ngv_h~dS+\sum_{E\in\mathcal{E}_h^e}\int_{E} \frac{\ep}{h_E} gv_h~dS
\end{align*}
for $w\in V$, $v_h\in V_h$
Moreover, $\Gamma_{\mathrm{in}}$ denotes the ``inflow'' boundary defined as  
\begin{equation}
 \label{eq:in}
\Gamma_{\mathrm{in}} = \Gamma_{\mathrm{in}} (t)=\{x\in
 \Gamma \mid b(x,t)\cdot n(x)<0\}.
\end{equation}
It is noteworthy that $\Gamma_{\mathrm{in}}$ is a time-dependent
region.

It is apparent that $a_{\ep,h}(t;\cdot,\cdot)$ is a bilinear form
on $V\times V_h$ for $t\in J$ and that $F_{\ep,h}(t;\cdot)$ is
a linear and continuous functional on $V_h$ for $t\in J$.

We have stated $a_{\ep,h}(t;\cdot,\cdot)$ in the case of $\mu=I$,
$b=0$, and $c=0$ presented in the Introduction. For a general $b$, we must add
the boundary integral term on $\Gamma_{\textrm{in}}$ to ensure the
coercivity of $a_{\ep,h}(t;\cdot,\cdot)$. Theorem
\ref{Thm:Coercivity_of_b} provides some related details. 

An alternate expression of \textbf{(P$_{\ep,h}$)} is presented below. 

\smallskip

\noindent\textbf{(P$_{\ep,h}$)} 
Find $u_{\ep,h}\in X_h$ such that
\begin{equation}
\label{Eq:Problem_Nitsche}
B_{\ep,h}(u_{\ep,h},\mathbf{y}_h) = \displaystyle\int_JF_{\ep,h}(t;y_h)~dt + (u_0,\widetilde{y}_h)_{L^2(\Omega)}\qquad (\mathbf{y}_h = (y_h,\widetilde{y}_h)\in Y_h),
\end{equation}
where $B_{\ep,h}$ denotes a bilinear form on $X_V\times Y_h$ defined as 
\begin{multline}
B_{\ep,h}(z,\mathbf{y}_h) =
 \displaystyle\int_{J}[(\partial_tz,y_h)_{L^2(\Omega)} +
 a_{\ep,h}(t;z,y_h)]~dt \\
 + (z(0),\widetilde{y}_h)_{L^2(\Omega)}\quad (z\in X_V,\mathbf{y}_h = (y_h,\widetilde{y}_h)\in Y_h).
\end{multline}
Hereinafter, we write $a_{\ep,h}(t;z,y_h)$ instead of
$a_{\ep,h}(t;z(t),y_h(t))$ for example. 

\subsection{Main results}
\label{sec:2c}

In this section, we state the main results presented in this paper, Theorems
\ref{Thm:Continuity_of_b}--\ref{Thm:ErrorEstimateinSpatialSemiDiscretization}.
For the penalty
parameter $\ep$, we make the following assumption.
\begin{assumption}
 \label{Assume:ep}
$\ep\ge 2\alpha^{-1}C^*\mu_1^2$.
\end{assumption}
Herein, the constants $\mu_1$ and $C^*$ appeared, respectively, in \eqref{eq:L4} and \eqref{Eq:Trace_Constant2}.
In the following theorems, we always presume that Assumptions
\ref{Assume:Ellipticity}--\ref{Assume:ep} are satisfied. 

Proofs of Theorems \ref{Thm:Continuity_of_b} and
\ref{Thm:Coercivity_of_b} will be explained in Section \ref{sec:4}.
We postpone presentation of the proofs of Theorems 
\ref{Thm:Continuity_of_B} and \ref{Thm:inf-sup_condition} for Section \ref{sec:5}. 
Theorem \ref{Thm:ErrorEstimateinSpatialSemiDiscretization} will be shown in Section \ref{sec:6}.
Other theorems are proved in this section. 

\begin{theorem}[Continuity of $a_{\ep,h}$]
\label{Thm:Continuity_of_b}
There exists a positive constant $C$ such that
\begin{equation}
a_{\ep,h}(t;w,v_h)\le C\|w\|_V\|v_h\|_{V_h}\quad(w\in V,\ v_h\in V_h,\ t\in J).
\end{equation}
In particular, there exists a positive constant $\widehat{M}$ such that
\begin{equation}
a_{\ep,h}(t;w_h,v_h)\le \widehat{M}\|w_h\|_{V_h}\|v_h\|_{V_h}\quad(w_h,v_h\in V_h,\ t\in J).
\end{equation}
Assumption \ref{Assume:ep} is not necessary for these inequalities to hold. 
\end{theorem}

\begin{theorem}[Coercivity of $a_{\ep,h}$]
\label{Thm:Coercivity_of_b}
There exists a positive constant $\widehat{\alpha}$ such that
\begin{equation}
a_{\ep,h}(t;v_h,v_h)\ge \widehat{\alpha}\|v_h\|_{V_h}^2\quad(v_h\in V_h,\ t\in J).
\end{equation}
\end{theorem}

\begin{theorem}[Continuity of $B_{\ep,h}$]
 \label{Thm:Continuity_of_B}
There exists a positive constant $C$ such that 
\begin{equation}
B_{\ep,h}(z,\mathbf{y}_h)\le C\|z\|_{X_{V}}\|\mathbf{y}_h\|_{Y_h}\quad (z\in X_{V},\ \mathbf{y}_h\in Y_h).
\end{equation}
Particularly, we have
\begin{equation}
\label{Eq:Continuityofbepsh}
B_{\ep,h}(z_h,\mathbf{y}_h)\le C\|z_h\|_{X_h}\|\mathbf{y}_h\|_{Y_h}\quad (z_h\in X_{h},\ \mathbf{y}_h\in Y_h).
\end{equation}
Assumption \ref{Assume:ep} is not necessary for these inequalities to hold. 
\end{theorem}

\begin{theorem}[Inf-sup condition of $B_{\ep,h}$]
\label{Thm:inf-sup_condition}
There exists a positive constant $\beta$ such that
\begin{equation}
\label{Eq:inf_sup_condi}
\inf_{0\neq z_h\in X_{h}}\sup_{0\neq \mathbf{y}_h\in Y_h}\dfrac{B_{\ep,h}(z_h,\mathbf{y}_h)}{\|z_h\|_{X_h}\|\mathbf{y}_h\|_{Y_h}}\ge\beta.
\end{equation}
\end{theorem}

\begin{theorem}
The problem \textup{\textbf{(P$_{\ep,h}$)}} has a unique solution $u_{\varepsilon,h}\in X_h$.
\end{theorem}

\begin{proof}
It is sufficient to prove that  
\begin{equation}
 \label{eq:bnb2}
B_{\ep,h}(z_h,\mathbf{y}_h)=0\quad(\forall z_h\in X_h)\quad \Rightarrow
\quad \mathbf{y}_h =(y_{h},\widetilde{y}_{h})=0.
\end{equation}
Actually, using \eqref{Eq:inf_sup_condi} and \eqref{eq:bnb2}, we can apply the Banach--Ne\v{c}as--Babu\v{s}ka theorem (see
 \cite[theorem2.6]{eg04} for example) to deduce the conclusion.  
First, presuming that $z_h\in X_h$ satisfies 
\begin{equation*}
z_h(0)=\widetilde{y}_{h},\mbox{ and }z_h(t)=0\quad (t\ge \delta)
\end{equation*}
for all $\delta>0$, then we have $\widetilde{y}_{h}=0$.

 To prove $y_h=0$, we take the basis functions $\{\phi_i\}_{i=1}^{N}$ of
 $V_h$, where $N:=\operatorname{dim}V_h$ and let
\begin{equation*}
y_{h}(t):=\displaystyle\sum_{i=1}^Na_i(t)\phi_i.
\end{equation*}
Then, $B_{\ep,h}\left(z_h,(y_{h},0)\right)=0$ implies
\begin{equation*}
\mathbf{A}_{\ep,h}(t)\mathbf{a}(t) =\mathbf{0},
\end{equation*}
where
 $\left(\mathbf{A}_{\ep,h}(t)\right)_{i,j}:=a_{\ep,h}(t;\phi_i,\phi_j)$
 and $\mathbf{a}(t) := (a_1(t), \dots, a_N(t))^{\mathrm{T}}$.
In view of the coercivity of $a_{\ep,h}(t;\cdot,\cdot)$ (Theorem
 \ref{Thm:Coercivity_of_b}), we obtain $\mathbf{a}(t)=\mathbf{0}$ for
 $t\in J$. This implies $y_{h}=0$; \eqref{eq:bnb2} is proved. 

\end{proof}

\begin{theorem}[Galerkin orthogonality]
\label{th:go}
 Letting $u_{\varepsilon,h}\in X_h$ be the solution of
 \textup{\textbf{(P$_{\ep,h}$)}}, then if the solution $u$ of
 \eqref{eq:1} satisfies $u\in X_V$, we have
\begin{equation}
B_{\ep,h}(u-u_{\ep,h},\mathbf{y}_h) = 0\quad(\mathbf{y}_h\in Y_h).
\end{equation}
\end{theorem}

\begin{proof}
Noting Lemma \ref{la:00}, we have
\begin{align*}
B_{\ep,h}(u-u_{\ep,h},\mathbf{y}_h) &= \displaystyle\int_J \left[(\partial_tu,y_h)_{L^2(\Omega)}+a_{\ep,h}(t;u,y_h)\right]~dt + (u(0),\widetilde{y}_h)_{L^2(\Omega)}\\
&\qquad - \displaystyle\int_JF_{\ep,h}(t;y_h)~dt + (u_0,\widetilde{y}_h)_{L^2(\Omega)}\\
&=\displaystyle\int_J\left[ (\partial_tu+L(t)u-f,y_h)_{L^2(\Omega)} - \displaystyle\sum_{E\in\mathcal{E}_h^e}\int_E(n\cdot\mu\nabla y_h)(u-g)~dS\right.\\
&\qquad\left. - \displaystyle\int_{\Gamma_{\mathrm{in}}}b\cdot n(u-g)y_h~dS + \displaystyle\sum_{E\in\mathcal{E}_h^e}\int_E\dfrac{\ep}{h_E}(u-g)y_h~dS~\right]dt\\
&=0
\end{align*}
for any $\mathbf{y}_h = (y_h,\widetilde{y}_h)\in Y_h$.
\end{proof}

\begin{theorem}[stability]
\label{th:st}
 Let $u_{\varepsilon,h}\in X_h$ be the solution of
 \textup{\textbf{(P$_{\ep,h}$)}}. If the solution $u$ of
 \eqref{eq:1} satisfies $u\in X_V$, then we have
\begin{equation}
\|u_{\ep,h}\|_{X_h} \le C\|u\|_{X_V}.
\end{equation}
\end{theorem}

\begin{proof}
Combining Theorems  \ref{Thm:Continuity_of_B},
 \ref{Thm:inf-sup_condition} and \ref{th:go}, we have 
 \[
    \|u_{\ep,h}\|_{X_h}
  \le \frac{1}{\beta}\sup_{\mathbf{y}_h\in Y_h}\frac{B_{\ep,h}(u_{\ep,h},\mathbf{y}_h)}{\|\mathbf{y}_h\|_{Y_h}}
  = \frac{1}{\beta}\sup_{\mathbf{y}_h\in
 Y_h}\frac{B_{\ep,h}(u,\mathbf{y}_h)}{\|\mathbf{y}_h\|_{Y_h}}\le \frac{1}{\beta}C\|u\|_{X_V}.
 \]

\end{proof}

\begin{theorem}[best approximation property]
\label{Lem:estimate_from_consistency}
If the solution $u$ of \eqref{eq:1} satisfies $u\in X_V$, then there exists a positive constant $C$ such that
\begin{equation}
\|u-u_{\varepsilon,h}\|_{X_h}\le C\inf_{z_h\in X_h}\|u-z_h\|_{X_V}.
\end{equation}
\end{theorem}

\begin{proof}
 In exactly the same way as the proof of Theorem \ref{th:st}, we have for
 any $z_h\in V_h$ 
\[
 \|z_h-u_{\ep,h}\|_{X_h} \le C\|z_h-u\|_{X_V}.
\]
This, together with the triangle inequality, implies the desired
 estimate. 

\end{proof}

\begin{theorem}[optimal order error estimate]
\label{Thm:ErrorEstimateinSpatialSemiDiscretization}
 Letting $l$ and $m$ be
 integers with $2\le l,m\le k+1$ and letting
 $u\in X_{l,m}:=H^1(J;H^{l}(\Omega),H^{m}(\Omega))$ be the solution of
 \eqref{eq:1}, we have
\begin{equation}
%\|u-u_{\varepsilon,h}\|_{X_h}^2\le C\left(\displaystyle\int_J(h^{2(l-1)}\|u\|_{H^{l}(\Omega)}^2+h^{2m}\|\partial_tu\|_{H^{m}(\Omega)}^2)~dt+h^{2j}\|u(0)\|_{H^{j}(\Omega)}^2\right),
\|u-u_{\varepsilon,h}\|_{X_h}^2\le C\left(h^{2(l-1)}\|u\|_{L^2(J;H^{l}(\Omega))}^2+h^{2m}\|\partial_tu\|_{L^2(J;H^{m}(\Omega))}^2+h^{2j}\|u(0)\|_{H^{j}(\Omega)}^2\right),
\end{equation}
where $j:=\min\{l,m\}$.
\end{theorem}

\section{Concrete examples of finite dimensional subspace}
\label{sec:3}

In this section, we give two concrete examples of the finite
dimensional subspace $V_h$ of $H^1(\Omega)$.

\subsection{Finite element method}

Letting $\Omega\subset\mathbb{R}^d$ be a polygonal domain and
introducing the triangulation $\mathcal{T}_h$ of $\Omega$, we consider
the standard continuous $\mathbb{P}^k$ finite element space
\begin{equation*}
V_h=\{v_h\in C^0(\overline{\Omega}) :
 v_h|_K\in\mathbb{P}_k\quad(K\in\mathcal{T}_h)\}.
\end{equation*}
It is readily apparent that $V_h\subset H^1(\Omega)$ and that Assumptions 
\ref{Assume:EqualityatExternalElement} and
\ref{Assume:Local_Regularity} are satisfied. 

Assuming that the family of triangulations $\{\mathcal{T}_h\}_h$ is
regular (\cite[(4.4.16)]{bs08} for example) with satisfaction of the inverse
assumption (\cite[(4.4.15)]{bs08} for example), then Assumption
\ref{Assume:inequalities} is satisfied. In summary, our results are
applicable to standard finite
element method.    

\subsection{Iso-Geometric Analysis}

Isogeometric analysis describes a computational domain by the so-called NURBS geometry. Furthermore, the finite dimensional subspace in the Galerkin method is introduced directly using the NURBS mesh. Here, we will review the definition and properties of NURBS.

\paragraph{Univariate B-spline basis functions on $[0,1]$}
We designate a vector $\Xi := \{\xi_1,\xi_2,\dots,\xi_r\}$ the \textit{knot vector} if
\begin{equation}
\xi_1 \leq \xi_2 \leq \dots \leq \xi_r.
\end{equation}
It is noteworthy that repetition of the knots is allowed. 
Without loss of generality, we let $\xi_1 = 0$ and $\xi_r = 1$.
Let $k$ be a given positive integer. Then, the univariate B-spline functions of degree $k$ associated with the knot vector $\Xi$ are defined by the Cox -- de Boor algorithm.
\begin{definition}
Let $\Xi = \{\xi_1, \dots, \xi_r\}$ be a knot vector. Then the $k$-th degree B-spline basis functions $\widehat{B}_{i,k}$ are defined as
\begin{equation}
\widehat{B}_{i,0}(\widehat{x}) = \left\{\begin{array}{ll}
1 &\mbox{ if }\xi_i\leq \widehat{x}\leq \xi_{i+1}\\
0 &\mbox{ otherwise}
\end{array}\right.\quad (k=0),
\end{equation}
\begin{equation}
\widehat{B}_{i,k}(\widehat{x}) = \frac{\widehat{x}-\xi_i}{\xi_{i+k}-\xi_i}\widehat{B}_{i,k-1}(\widehat{x}) + \frac{\xi_{i+k+1}-\widehat{x}}{\xi_{i+k+1}-\xi_{i+1}}\widehat{B}_{i+1,k-1}(\widehat{x})\quad (k\geq1),
\end{equation}
with $i = 1, \dots, r-k-1$. Here, $0/0 = 0$ should be replaced by $0$ in this definition.  
\end{definition}
We state some properties of the B-spline basis functions of degree $k$. They are non-negative $k$-th degree piecewise polynomials such that $\widehat{B}_{i,k}(\widehat{x})=0$ for $\widehat{x}\not\in[\xi_i,\xi_{i+k+1}]$. Now we introduce an alternative representation of $\Xi$ to state the other properties. Let
\begin{equation}
\Xi = \{\underbrace{\zeta_0,\dots,\zeta_0}_{m_0\mbox{ times}},\underbrace{\zeta_1\dots,\zeta_1}_{m_1\mbox{ times}},\dots,\underbrace{\zeta_N,\dots,\zeta_N}_{m_N\mbox{ times}}\},
\end{equation}
where $\zeta_0< \zeta_1< \dots< \zeta_N$. Therein we designate the multiplicity of $\zeta_n$ by $m_n$. 
Assume that $m_n\leq k+1$ for all knots, then $\widehat{B}_{i,k}$ has $k-m_n$ continuous derivatives at internal node $\zeta_n$. 
Furthermore, one can say that the knot vector $\Xi$ is $k$-open if $m_0 = m_N = k+1$. 
Letting $\Xi$ be a $k$-open knot vector, then $\widehat{B}_{i,k}$ form the partition of unity. They also form the basis of \textit{spline space}, i.e., the space of piecewise polynomials of degree $k$ with $k-m_n$ continuous derivatives at $\zeta_n$ for $n = 1, \dots, N-1$.
 
Henceforth, we assume the knot vector $\Xi$ is $k$-open. We define the \textit{univariate spline} $S_k(\Xi)$ as
\begin{equation}
S_k(\Xi) = \mathrm{span}\{\widehat{B}_{i,k}\mid i = 1, \dots, r-k-1\}.
\end{equation}

For partition size $h_n = \zeta_n-\zeta_{n-1}$, the following assumption is needed.
\begin{assumption}[Local quasi-uniform]
\label{local_quasi-uniformity}
The knot vector $\Xi$ is locally quasi-uniform, i.e., there exists a constant $\theta \geq 1$ such that
\begin{equation}
\dfrac{1}{\theta}\leq \theta_n = \dfrac{h_n}{h_{n+1}}\leq \theta\quad(n=1,\cdots, N-1).
\end{equation}
\end{assumption}

Here we mention that the \textit{quasi-interpolant operator} $\Pi_{k,\Xi}: L^\infty(I)\to S_k(\Xi)$ satisfies the error estimate (as described in greater detail in \cite[Chapter 4]{schum07}). Letting
\begin{equation}
I_n = (\zeta_{n-1}, \zeta_n), \ h_n = |I_n|,\ \widetilde{I}_n = \bigcup\left\{\mathrm{supp}\widehat{B}_{i,p}\mid \widehat{B}_{i,p}|_{I_n} \not\equiv 0\right\}, \ \widetilde{h}_n = |\widetilde{I}_n|,
\end{equation}
then the following estimate holds.
\begin{lemma}[Error estimate]
Let $s$ be a positive integer, $p\in[1,\infty]$ and $l=\min\{k+1,s\}$. Then there exists a positive constant $C$ such that
\begin{equation}
\|f-\Pi_{k,\Xi}(f)\|_{L^p(I_n)} \leq C\widetilde{h}_n^{l}|f|_{W^{l,p}(\widetilde{I}_n)}\quad(f\in W^{s,p}(I),\ n=1,\cdots, N).
\end{equation}
Moreover, let Assumption \ref{local_quasi-uniformity} be satisfied and let $m$ be an integer with $0\le m\le l$. Then there exists a constant $C$ such that
\begin{equation}
|f-\Pi_{k,\Xi}(f)|_{W^{m,p}(I_n)}\leq C\widetilde{h}_n^{l-m}|f|_{W^{l,p}(\widetilde{I}_n)}\quad(f\in W^{s,p}(I),\ n=1,\cdots, N).
\end{equation}
\end{lemma}
The proof can be found in \cite[Proposition]{bbsv14} for $p=2$. We refer the reader to \cite[Theorem4.41]{schum07} for general $p$.

\paragraph{Multivariate B-spline basis functions and NURBS basis functions}
Let $\widehat{d}$ be the space dimension with $\widehat{d}\le d$. For $j=1,\dots,\widehat{d}$, given degree $k_j$ and $k_j$-open and locally quasi-uniform knot vector 
\begin{align*}
\Xi_j &= \{\xi_{j,1}, \dots, \xi_{j,r_j}\}\\
&= \{\underbrace{\zeta_{j,0},\dots,\zeta_{j,0}}_{k_j+1\mbox{ times}},\underbrace{\zeta_{j,2},\dots,\zeta_{j,2}}_{m_{j,2}\mbox{ times}},\dots,\underbrace{\zeta_{j,N_j},\dots,\zeta_{j,N_j}}_{k_j+1\mbox{ times}}\},
\end{align*}
we can define the $k_j$-th degree univariate B-spline basis functions
\begin{equation}
\widehat{B}_{i_j,k_j}(\widehat{x}_j),\quad i_j = 1, 2, \dots, r_j-k_j-1.
\end{equation}

Furthermore, the knots without repetition provide the mesh on a parametric domain $\widehat{\Omega} = (0,1)^{\widehat{d}}$, which is denoted as $\widehat{\mathcal{T}}_h$:
\begin{equation}
\widehat{\mathcal{T}}_h =\{I_{1,n_1}\times \dots \times I_{\widehat{d},n_{\widehat{d}}} \mid I_{j,n_j} := (\zeta_{j,n_j-1}, \zeta_{j,n_j}), n_j = 1, \dots, N_j\}.
\end{equation}
Then we define the multivariate B-spline basis functions as
\begin{equation}
\widehat{\mathbf{B}}_{\mathbf{i},\mathbf{k}}(\widehat{\mathbf{x}}) = \widehat{B}_{i_1,k_1}(\widehat{x}_1) \cdots \widehat{B}_{i_{\widehat{d}},k_{\widehat{d}}}(\widehat{x}_{\widehat{d}})
\end{equation}
for $\mathbf{i} = (i_1, \dots, i_{\widehat{d}})$, where $\mathbf{k} = (k_1, \dots, k_{\widehat{d}})$ and $\widehat{\mathbf{x}} = (\widehat{x}_1,\dots,\widehat{x}_{\widehat{d}})\in \widehat{\Omega}$.
We define the \textit{multivariate spline} $S_{\mathbf{k}}(\boldsymbol{\Xi})$ as
\begin{equation}
S_{\mathbf{k}}(\boldsymbol{\Xi}) = S_{k_1}(\Xi_1)\otimes\dots\otimes S_{k_{\widehat{d}}}(\Xi_{\widehat{d}}) = \operatorname{span}\{\widehat{\mathbf{B}}_{\mathbf{i},\mathbf{k}}\mid \mathbf{i}\in\mathbf{I}\},
\end{equation}
where $\boldsymbol{\Xi} = \Xi_1 \times \dots \times \Xi_{\widehat{d}}$ and $\mathbf{I} = \{\mathbf{i}=(i_1,\dots,i_{\widehat{d}})\mid i_j=1,\dots,r_j-k_j-1\}$.
The quasi-interpolation for multivariate B-spline is defined also by the tensor product:
\begin{equation}
\Pi_{\mathbf{k},\boldsymbol{\Xi}} = \Pi_{k_1,\Xi_1}\otimes\dots\otimes\Pi_{k_{\widehat{d}},\Xi_{\widehat{d}}}: L^{\infty}(\widehat{\Omega})\to S_{\mathbf{k}}(\boldsymbol{\Xi}).
\end{equation}

%\begin{remark}
%In general, we can not find the stability of $\Pi_{\mathbf{k},\boldsymbol{\Xi}}$ because $L^{\infty}((0,1)^2) \neq L^{\infty}(0,1)\times L^{\infty}(0,1)$
%\end{remark}

Here, the definition of NURBS basis functions for given weight
\begin{equation}
W(\widehat{\mathbf{x}}) = \displaystyle\sum_{\mathbf{j}\in \mathbf{I}}w_{\mathbf{j}}\widehat{\mathbf{B}}_{\mathbf{j},\mathbf{k}}(\widehat{\mathbf{x}})
%\{w_{\mathbf{i}} : \mathbf{i}\in\mathbf{I}\}
\end{equation}
is described as
\begin{equation}
\widehat{\mathbf{N}}_{\mathbf{i},\mathbf{k}}(\widehat{\mathbf{x}}) = \dfrac{w_{\mathbf{i}}\widehat{\mathbf{B}}_{\mathbf{i},\mathbf{k}}(\widehat{\mathbf{x}})}{W(\widehat{\mathbf{x}})}\quad(\textbf{i}\in \textbf{I}),
\end{equation}
where positive constants $w_{\mathbf{j}}>0$ ($\mathbf{j}\in\mathbf{I}$) are called weights.
Furthermore, a NURBS parametrization is given by a linear combination of NURBS basis functions. Letting $P_{\mathbf{i}}\in\mathbb{R}^{d}$ be control points, then a NURBS parametrization $\mathbf{F}(\widehat{\mathbf{x}})$ is given as
\begin{equation}
\mathbf{F}(\widehat{\mathbf{x}}) = \displaystyle\sum_{\mathbf{i}\in\mathbf{I}}P_{\mathbf{i}}\widehat{\mathbf{N}}_{\mathbf{i},\mathbf{k}}(\widehat{\mathbf{x}}).
\end{equation}
This parametrization can define the NURBS geometry in $\mathbb{R}^d$. In this paper, we only consider when $\widehat{d}=d$.

The requirement on the map $\mathbf{F}$ is that it satisfy the following regularity.

\begin{assumption}
\label{Assume:ByLipschitz}
The map $\mathbf{F}: \widehat{\Omega}\to\Omega$ is bijective Lipschitz function whose inverse function is also Lipschitz. Moreover, $\mathbf{F}|_{\overline{Q}}\in C^{\infty}(\overline{Q})$ and $\mathbf{F}^{-1}|_{\overline{K}}\in C^{\infty}(\overline{K})$ for all $Q\in \widehat{\mathcal{T}}_h$, $K\in\mathcal{T}_h$.  
\end{assumption}

\begin{definition}
\label{Def:NURBS_Domain}
Under Assumption \ref{Assume:ByLipschitz}, the domain $\Omega$ defined as $\Omega = \textbf{F}(\widehat{\Omega})$ is called the NURBS domain.  
\end{definition}

\begin{figure}[tb]
\centering
\includegraphics[clip,scale=.9]{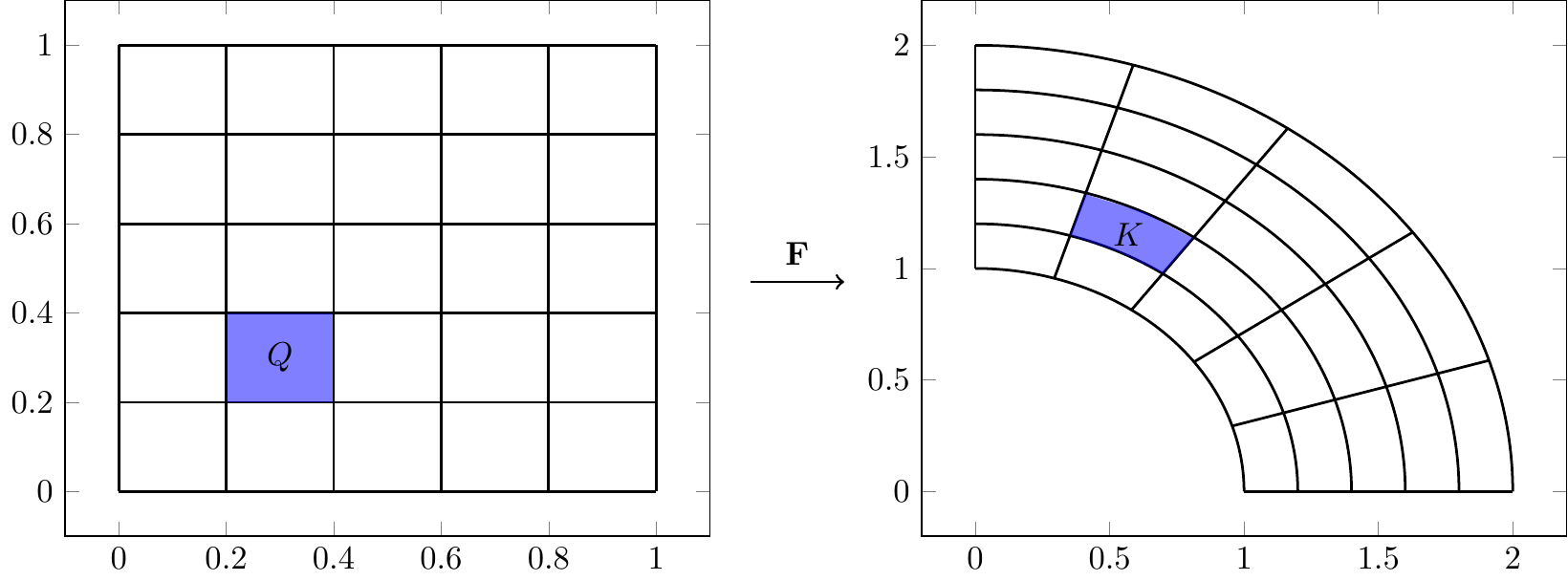}
\caption{Parametric mesh and the mesh of physical domain $\Omega$.}
\end{figure}

A mesh on $\Omega$ is provided as the image of the parametric mesh as
\begin{equation}
\mathcal{T}_h = \{K = \mathbf{F}(Q) \mid Q\in\widehat{\mathcal{T}}_h\}.
\end{equation}

Under Assumption \ref{Assume:ByLipschitz}, we define
\begin{equation}
\label{discrete_space_Rh}
V_h = \mbox{span}\{\mathbf{N}_{\mathbf{i},\mathbf{k}}(\mathbf{x}) = \widehat{\mathbf{N}}_{\mathbf{i},\mathbf{k}}\circ \mathbf{F}^{-1}(\mathbf{x})\mid \mathbf{i}\in\mathbf{I}\},
\end{equation}
where $h$ is the mesh size $h=\max\{h_Q=\mbox{diam}(Q) \mid Q\in \widehat{\mathcal{T}}_h\}$. Furthermore, we define
\begin{equation}
h_K=\|\nabla\mathbf{F}\|_{L^{\infty}(Q)}h_Q\quad(K\in\mathcal{T}_h),
\end{equation}
where $Q=\mathbf{F}^{-1}(K)$. For the NURBS mesh, we define the regularity of the family of mesh $\{\mathcal{T}_h\}_h$ using $\{\widehat{\mathcal{T}}_h\}$.

\begin{assumption}
\label{Assume:Regularity_in_IGA}
The family of the mesh $\{\mathcal{T}_h\}_h$ is regular, meaning that there exists a positive constant $\sigma$ such that
\begin{equation}
\dfrac{h_Q}{h_{Q,\mathrm{min}}}\le \sigma\quad \quad\left(Q\in\bigcup_h\widehat{\mathcal{T}}_h\right),
\end{equation}
where $h_{Q,\mathrm{min}}$ represents the length of the smallest edge of hypercube $Q$.
\end{assumption}

We always assume that Assumptions \ref{local_quasi-uniformity}, \ref{Assume:ByLipschitz} and \ref{Assume:Regularity_in_IGA} are satisfied. Now we review some results obtained in earlier studies.

\begin{lemma}[Trace inequality, Theorem 3.2 of \cite{eh13}]
Letting $K\in \mathcal{T}_h$ and $Q = \mathbf{F}^{-1}(K)$, then
\begin{equation}
\|f\|_{L^2(\partial K)}^2\le C\lambda_Q\lambda_K\left(h_K^{-1}\|f\|_{L^2(K)}^2+h_K|f|_{H^1(K)}^2\right)\quad (f\in H^1(K)),
\end{equation}
where $\lambda_Q$ and $\lambda_K$ respectively represent the local shape regularity constants of $Q$ and $K$. They are independent of $h_K$.
\end{lemma}

\begin{lemma}[Inverse inequality, Theorem 4.2. of \cite{bbchs06}]
Letting $l$, $m$ be integers with $0\le m\le l$, we have
\begin{equation}
\|v_h\|_{H^{l}(K)}\le C_{\mathrm{shape}}h_K^{m-l}\displaystyle\sum_{i=0}^m\|\nabla\mathbf{F}\|_{L^{\infty}(\mathbf{F}^{-1}(K))}^{i-m}|v_h|_{H^{i}(K)},\quad(K\in \mathcal{T}_h, v_h\in V_h).
\end{equation}
Especially, we have
\begin{equation}
|v_h|_{H^1(K)}\le \|v_h\|_{H^1(K)}\le C_{\mathrm{shape}}h_K^{-1}\|v_h\|_{L^2(K)}\quad(K\in \mathcal{T}_h, v_h\in V_h).
\end{equation}
\end{lemma}

\begin{lemma}[Quasi-interpolation error estimate, Corollary 4.21 of \cite{bbsv14}]
Let the projection $\Pi_{V_h}:L^2(\Omega)\to V_h$ be
\begin{equation}
\Pi_{V_h}f = \dfrac{\Pi_{\mathbf{k},\boldsymbol{\Xi}}(W(f\circ\mathbf{F}))}{W}\circ\mathbf{F}^{-1}\quad (f\in L^2(\Omega)).
\end{equation}
Furthermore, letting $s$ be an integer, $l = \min\{k_1+1,\cdots, k_d+1,s\}$ and $0\le m\le l$, then there exists a positive constant $C$ such that
\begin{equation}
\|v-\Pi_{V_h}v\|_{H^m(K)}\le Ch_{\widetilde{K}}^{l-m}\|v\|_{H^\ell(\widetilde{K})}\quad(K\in\mathcal{T}_h, v\in H^s(\Omega)),
\end{equation}
where $\widetilde{K}= \mathbf{F}(\widetilde{Q})$ for $Q=\mathbf{F}^{-1}(K)$ and 
\begin{equation}
\widetilde{Q} = \bigcup\left\{\operatorname{supp}\widehat{\mathbf{N}}_{\mathbf{i},\mathbf{k}}: \widehat{\mathbf{N}}_{\mathbf{i},\mathbf{k}}|_{Q} \not\equiv 0\right\}.
\end{equation}
\end{lemma}

In summary, Assumptions \ref{Assume:EqualityatExternalElement} -- \ref{Assume:inequalities} are satisfied under Assumptions \ref{local_quasi-uniformity} -- \ref{Assume:Regularity_in_IGA}.

%\section{Proof of Theorem \ref{Thm:Estimate_for_Normal_Derivative},\ref{Thm:Continuity_of_b} and \ref{Thm:Coercivity_of_b}}
\section{Proof of Theorem \ref{Thm:Continuity_of_b} and \ref{Thm:Coercivity_of_b}}
\label{sec:4}

This section is devoted to the proof of theorems for the ``elliptic part'' $a_{\ep,h}$. 
We start with the following auxiliary lemma: 
the estimate \eqref{Eq:Normal_Derivative_Bounded_by_H1} itself is well
known (see for instance \cite[Lemma 2.1]{tho06}); the estimate
\eqref{Eq:Normal_Derivative_Bounded_by_H2} is apparently unfamiliar.

\begin{lemma}
\label{Lem:Estimate_for_Normal_Derivative}
There exists a positive constant $C$ such that
\begin{equation}
\label{Eq:Normal_Derivative_Bounded_by_H2}
\displaystyle\sum_{E\in\mathcal{E}_h^e}h_E\|n\cdot\mu\nabla v\|_{L^2(E)}^2 \le C\left(|v|_{H^1(\Omega)}^2 + \sum_{K\in\mathcal{T}_h}h_K^2|v|_{H^2(K)}^2\right)\quad(v\in V).
\end{equation}
Particularly, there exists a positive constant $C_I$ such that
\begin{equation}
\label{Eq:Normal_Derivative_Bounded_by_H1}
\displaystyle\sum_{E\in\mathcal{E}_h^e}h_E\|n\cdot\mu\nabla v_h\|_{L^2(E)}^2 \le C_I|v_h|_{H^1(\Omega)}^2\quad(v_h\in V_h).
\end{equation}
\end{lemma}

\begin{proof}
Because $\mu\in\mathbb{R}^{d\times d}$ is symmetric, we have
\begin{equation}
\sup_{\xi\in\mathbb{R}^d}\dfrac{|\mu\xi|}{|\xi|} = \sup_{\xi\in\mathbb{R}^d}\dfrac{\mu\xi\cdot\xi}{|\xi|^2}\le \mu_1.
\end{equation}
This result implies $|n\cdot\mu\nabla v|\le \mu_1|\nabla v|$. Therefore
\begin{equation}
\|n\cdot\mu\nabla v\|_{L^2(E)}^2 \le \mu_1^2\displaystyle\sum_{i=1}^d\left\|\operatorname{Tr}\left(\dfrac{\partial v}{\partial x_i}\right)\right\|_{L^2(E)}^2
\end{equation}
for all $E\in \mathcal{E}_h^e$. Here Assumption \ref{Assume:inequalities} (i) yields
\begin{equation}
\mu_1^2\displaystyle\sum_{i=1}^d\left\|\operatorname{Tr}\left(\dfrac{\partial v}{\partial x_i}\right)\right\|_{L^2(E)}^2 \le C_{\mathrm{Tr}}\mu_1^2h_E^{-1}(|v|_{H^1(K_E)}^2 + h_{K_E}^2|v|_{H^2(K_E)}^2)\quad(v\in V,~ E\in \mathcal{E}_h^e),
\end{equation}
which implies \eqref{Eq:Normal_Derivative_Bounded_by_H2} with $C=C_{\mathrm{Tr}}\mu_1^2$. Furthermore, using equation \eqref{Eq:Trace_Constant2} leads to the following:
\begin{equation*}
\|n\cdot\mu\nabla v_h\|_{L^2(E)}^2 \le \mu_1^2\displaystyle\sum_{i=1}^d\left\|\operatorname{Tr}\left(\dfrac{\partial v_h}{\partial x_i}\right)\right\|_{L^2(E)}^2\le C^*\mu_1^2h_E^{-1}|v_h|_{H^1(K_E)}^2\quad(v_h\in V_h,~ E\in\mathcal{E}_h^e).
\end{equation*}
Then we have the estimate \eqref{Eq:Normal_Derivative_Bounded_by_H1} with $C_I = C^*\mu_1^2$.
\end{proof}

Next we can state the following proofs.

\begin{proof}[Proof of Theorem \ref{Thm:Continuity_of_b}]
First, the Cauchy--Schwarz inequality gives
\begin{align*}
a_{\ep,h}(t;w,v_h) &\le M\|w\|_{H^1(\Omega)}\|v_h\|_{H^1(\Omega)}+\left(\displaystyle\sum_{E\in\mathcal{E}_h^e}h_E\|n\cdot\mu\nabla w\|_{L^2(E)}^2\right)^{1/2}\left(\sum_{E\in\mathcal{E}_h^e}h_E^{-1}\|v_h\|_{L^2(E)}^2\right)^{1/2}\\
&\quad +\left(\displaystyle\sum_{E\in\mathcal{E}_h^e}h_E\|n\cdot\mu\nabla v_h\|_{L^2(E)}^2\right)^{1/2}\left(\sum_{E\in\mathcal{E}_h^e}h_E^{-1}\|w\|_{L^2(E)}^2\right)^{1/2} + C\|w\|_{L^2(\Gamma)}\|v_h\|_{L^2(\Gamma)}\\
&\qquad + \ep\left(\displaystyle\sum_{E\in\mathcal{E}_h^e}h_E^{-1}\|w\|_{L^2(E)}^2\right)^{1/2}\left(\displaystyle\sum_{E\in\mathcal{E}_h^e}h_E^{-1}\|v_h\|_{L^2(E)}^2\right)^{1/2}
\end{align*}
for all $w\in V$, $v_h\in V_h$ and $t\in J$. By applying the (standard) trace inequality and Lemma \ref{Lem:Estimate_for_Normal_Derivative}, we have
\begin{equation}
a_{\ep,h}(t;w,v_h) \le C\|w\|_V\|v_h\|_{V_h}\quad(w\in V,~ v_h\in V_h,~ t\in J).
\end{equation}
Moreover, the definition of norm $\|\cdot\|_V$ and $\|\cdot\|_{V_h}$ yields that there exists a positive constant $\widehat{M}$ such that
\begin{equation}
a_{\ep,h}(t;w_h,v_h)\le C\|w_h\|_V\|v_h\|_{V_h}\le \widehat{M}\|w_h\|_{V_h}\|v_h\|_{V_h}\quad(w_h,v_h\in V_h,~ t\in J),
\end{equation}
which is the desired inequality.
\end{proof}

\begin{proof}[Proof of Theorem \ref{Thm:Coercivity_of_b}]
Fix $v_h\in V_h$ and $t\in J$ arbitrarily. First, we mention that $V_h\not\subset H^1_0(\Omega)$. Therefore
\begin{equation}
a(t;v_h,v_h) \ge \alpha\|v_h\|_{H^1(\Omega)}^2 + \dfrac{1}{2}(b\cdot nv_h,v_h)_{L^2(\Gamma_{\mathrm{in}})},
\end{equation}
and $(b\cdot nv_h,v_h)_{\Gamma_{\mathrm{in}}} < 0$. Now we have the following:
\begin{align*}
a_{\ep,h}(t;v_h,v_h) &\ge \alpha\|v_h\|_{H^1(\Omega)}^2 + \dfrac{1}{2}(b\cdot nv_h,v_h)_{L^2(\Gamma_{\mathrm{in}})} - 2\displaystyle\sum_{E\in\mathcal{E}_h^e}(n\cdot\mu\nabla v_h,v_h)_{L^2(E)}\\
&\quad - (b\cdot nv_h,v_h)_{L^2(\Gamma_{\mathrm{in}})} + \ep\displaystyle\sum_{E\in\mathcal{E}_h^e}h_E^{-1}\|v_h\|_{L^2(E)}^2\\
&\ge \alpha\|v_h\|_{H^1(\Omega)}^2 - \dfrac{1}{2}(b\cdot nv_h,v_h)_{L^2(\Gamma_{\mathrm{in}})} + \ep\displaystyle\sum_{E\in\mathcal{E}_h^e}h_E^{-1}\|v_h\|_{L^2(E)}^2\\
&\quad - 2\left(\sum_{E\in\mathcal{E}_h^e}h_E\|n\cdot\mu\nabla v_h\|_{L^2(E)}^2\right)^{1/2}\left(\sum_{E\in\mathcal{E}_h^e}h_E^{-1}\|v_h\|_{L^2(E)}^2\right)^{1/2}.
\end{align*}
Recalling $-(b\cdot nv_h,v_h)_{L^2(\Gamma_{\mathrm{in}})}>0$ and Lemma \ref{Lem:Estimate_for_Normal_Derivative}, then we have
\begin{equation*}
a_{\ep,h}(t;v_h,v_h)\ge \alpha\|v_h\|_{H^1(\Omega)}^2 + \ep\sum_{E\in\mathcal{E}_h^e}h_E^{-1}\|v_h\|_{L^2(E)}^2- 2C_I^{1/2}|v_h|_{H^1(\Omega)}\left(\sum_{E\in\mathcal{E}_h^e}h_E^{-1}\|v_h\|_{L^2(E)}^2\right)^{1/2}.
\end{equation*}
Here we apply Young's inequality to obtain
\begin{equation}
2C_I^{1/2}|v_h|_{H^1(\Omega)}\left(\displaystyle\sum_{E\in\mathcal{E}_h^e}h_E^{-1}\|v_h\|_{L^2(E)}^2\right)^{1/2} \le \dfrac{\alpha}{2}|v_h|_{H^1(\Omega)}^2 + 2\alpha C_I\sum_{E\in\mathcal{E}_h^e}h_E^{-1}\|v_h\|_{L^2(E)}^2.
\end{equation}
Then we have
\begin{equation}
a_{\ep,h}(t;v_h,v_h) \ge \dfrac{\alpha}{2}\|v_h\|_{H^1(\Omega)}^2 + (\ep-2\alpha C_I)\sum_{E\in\mathcal{E}_h^e}h_E^{-1}\|v_h\|_{L^2(E)}^2.
\end{equation}
Letting $\widehat{\alpha} = \min\{\dfrac{\alpha}{2},\ep-2\alpha C_I\}$ yields the desired conclusion.
\end{proof}

\section{Proof of Theorem \ref{Thm:Continuity_of_B} and \ref{Thm:inf-sup_condition}}
\label{sec:5}
This section presents that $B_{\ep,h}$ is continuous and satisfies the inf-sup condition. First, we show Theorem \ref{Thm:Continuity_of_B}.

\begin{proof}[Proof of Theorem \ref{Thm:Continuity_of_B}]
Recalling that every $z\in X_V$ and $z_h\in X_h\subset C^0(\overline{J};V_h)$ satisfy $z(t)\in V$, $z_h(t)\in V_h$ for any $t\in J$, we can apply Theorem \ref{Thm:Continuity_of_b} to obtain
\begin{align*}
B_{\ep,h}(z,\mathbf{y}_h) &= \int_J((\partial_tz,y_h)_{L^2(\Omega)} + a_{\ep,h}(t;z,y_h))~dt + (z(0),\widetilde{y}_h)_{L^2(\Omega)}\\
&\le \int_J(\|\partial_tz\|_{V_h'}\|y_h\|_{V_h}+C\|z\|_{V}\|y_h\|_{V_h})~dt + \|z(0)\|_{L^2(\Omega)}\|\widetilde{y}_h\|_{L^2(\Omega)}\\
&\le C\|z\|_{X_V}\|\mathbf{y}_h\|_{Y_h}\quad (z\in X_V,\ \mathbf{y}_h = (y_h,\widetilde{y}_h)\in Y_h),
\end{align*}
which is the first desired estimate. Moreover, $\|z_h\|_{X_V} \le C\|z_h\|_{X_h}$ leads to the second inequality. 
\end{proof}

Stating the proof of Theorem \ref{Thm:inf-sup_condition} requires an auxiliary operator $\mathcal{A}_{\ep,h}(t):V_h\to V_h$ defined by setting
\begin{equation}
(\mathcal{A}_{\ep,h}(t)w_h,v_h)_{L^2(\Omega)} = a_{\ep,h}(t;w_h,v_h) \quad(w_h,v_h\in V_h,\ t\in J).
\end{equation}
We recall that the bilinear form $a_{\ep,h}(t;\cdot,\cdot):V_h\times V_h\to\mathbb{R}$ is continuous and coercive. Therefore, the Lax--Milgram theorem shows that the operator $\mathcal{A}_{\ep,h}(t)$ is invertible for $t\in J$. Now we have the following lemma.

\begin{lemma}
\label{Lem:Estimates_on_A_ep_h_inv}
Operator $\mathcal{A}_{\ep,h}(t)^{-1}:V_h\to V_h$ satisfies
\begin{align}
\|\mathcal{A}_{\ep,h}(t)^{-1}v_h\|_{V_h}&\le \widehat{\alpha}^{-1}\|v_h\|_{V_h'},\\
(v_h,\mathcal{A}_{\ep,h}(t)^{-1}v_h)_{L^2(\Omega)}&\ge \widehat{\alpha}\widehat{M}^{-2}\|v_h\|_{V_h'}^2
\end{align}
for $v_h\in V_h$ and $t\in J$.
\end{lemma}

\begin{proof}
First, we have
\begin{align*}
\widehat{\alpha}\|\mathcal{A}_{\ep,h}(t)^{-1}v_h\|^2_{V_h}&\le (\mathcal{A}_{\ep,h}(t)(\mathcal{A}_{\ep,h}(t)^{-1}v_h),\mathcal{A}_{\ep,h}(t)^{-1}v_h)_{L^2(\Omega)}\\
&= (v_h,\mathcal{A}_{\ep,h}(t)^{-1}v_h)_{L^2(\Omega)}
\le \|v_h\|_{V_h'}\|\mathcal{A}_{\ep,h}(t)^{-1}v_h\|_{V_h}
\end{align*}
for all $v_h\in V_h$. Therefore we have
\begin{equation}
\|\mathcal{A}_{\ep,h}(t)^{-1}v_h\|_{V_h}\le \widehat{\alpha}^{-1}\|v_h\|_{V_h'}\quad (v_h\in V_h,\ t\in J).
\end{equation}

Next, it is noteworthy that the constant $\widehat{M}$ satisfies
\begin{equation*}
\sup_{0\neq w_h\in V_h}\dfrac{\|\mathcal{A}_{\ep,h}(t)w_h\|_{V_h'}}{\|w_h\|_{V_h}} = \sup_{0\neq w_h\in V_h}\sup_{0\neq v_h\in V_h}\dfrac{(\mathcal{A}_{\ep,h}(t)w_h,v_h)_{L^2(\Omega)}}{\|w_h\|_{V_h}\|v_h\|_{V_h}}\le \widehat{M}.
\end{equation*}
Therefore,
\begin{equation}
\|v_h\|_{V_h'} = \|\mathcal{A}_{\ep,h}(t)(\mathcal{A}_{\ep,h}(t)^{-1}v_h)\|_{V_h'}\le\widehat{M}\|\mathcal{A}_{\ep,h}(t)^{-1}v_h\|_{V_h}
\end{equation}
for all $v_h\in V_h$. This yields
\begin{align*}
(v_h,\mathcal{A}_{\ep,h}(t)^{-1}v_h)_{L^2(\Omega)}&=(\mathcal{A}_{\ep,h}(t)(\mathcal{A}_{\ep,h}(t)^{-1}v_h),\mathcal{A}_{\ep,h}(t)^{-1}v_h)_{L^2(\Omega)}\\
&\ge \widehat{\alpha}\|\mathcal{A}_{\ep,h}(t)^{-1}v_h\|^2_{V_h}\\
&\ge \widehat{\alpha}\widehat{M}^{-2}\|v_h\|_{V_h'}^2\quad (v_h\in V_h,\ t\in J).
\end{align*}
\end{proof}

We can state the following proof.

\begin{proof}[Proof of Theorem \ref{Thm:inf-sup_condition}]
After fixing $z_h\in X_h$ arbitrarily, let $\mathbf{y}_h:= (\mathcal{A}_{\ep,h}(t)^{-1}(\partial_tz_h)+\delta z_h, \delta z_h(0))\in Y_h$, where $\delta := \widehat{\alpha}^{-4}\widehat{M}^4$. Then we have
\begin{align*}
\|\mathbf{y}_h\|_{Y_h}^2 &= \displaystyle\int_J\|\mathcal{A}_{\varepsilon}(t)^{-1}(\partial_tz_h)+\delta z_h\|_{V_h}^2~dt + \|\delta z_h(0)\|_{L^2(\Omega)}^2\\
&\le \max\{\widehat{\alpha}^{-1},\delta\}\displaystyle\int_J(\partial_tz_h\|_{V_h'}^2 + \|z_h\|_{V_h}^2)~dt +\delta\|z_h(0)\|_{L^2(\Omega)}^2 \le \max\{\widehat{\alpha}^{-1},\delta\}\|z_h\|_{X_h}^2,
\end{align*}
and
\begin{align*}
B_{\ep,h}(z_h,\mathbf{y}_h) &= \displaystyle\int_J\left((\partial_tz_h,\mathcal{A}_{\ep,h}(t)^{-1}(\partial_tz_h)+\delta z_h)_{L^2(\Omega)} +a_{\ep,h}(t;z_h,\mathcal{A}_{\ep,h}(t)^{-1}(\partial_tz_h)+\delta z_h)\right)~dt\\
&\qquad + (z_h(0),\delta z_h(0))_{L^2(\Omega)}\\
&\ge\displaystyle\int_J(\widehat{\alpha}\widehat{M}^{-2}\|\partial_tz_h\|_{V_h'}^2 - \widehat{\alpha}^{-1}\widehat{M}\|z_h\|_{V_h}\|\partial_tz_h\|_{V_h'} + \delta\widehat{\alpha}\|z_h\|_{V_h}^2)~dt\\
&\qquad + \delta/2(\|z_h(T)\|_{L^2(\Omega)}^2 - \|z_h(0)\|_{L^2(\Omega)}^2)+ \delta\|z_h(0)\|_{L^2(\Omega)}^2\\
&\ge \displaystyle\int_J (\widehat{\alpha}\widehat{M}^{-2}\|\partial_tz_h\|_{V_h'}^2 - \widehat{\alpha}^{-1}\widehat{M}\|z_h\|_{V_h}\|\partial_tz_h\|_{V_h'} + \delta\widehat{\alpha}\|z_h\|_{V_h}^2)~dt\\
&\qquad + \delta/2\|z_h(0)\|_{L^2(\Omega)}^2\\
&\ge\displaystyle\int_J \left(\widehat{\alpha}\widehat{M}^{-2}/2 \|\partial_tz_h\|_{V_h'}^2 + (\delta\widehat{\alpha} - \widehat{\alpha}^{-3}\widehat{M}^4/2)\|z_h\|_{V_h}^2\right)~dt + \delta/2\|z_h(0)\|_{L^2(\Omega)}^2\\
&\ge 1/2\min\{\widehat{\alpha}\widehat{M}^{-2},\widehat{\alpha}^{-3}\widehat{M}^4,\widehat{\alpha}^{-4}\widehat{M}^4\}\left(\displaystyle\int_J(\|\partial_tz_h\|_{V_h'}^2 + \|z_h\|_{V_h}^2)~dt + \|z_h(0)\|_{L^2(\Omega)}^2\right)\\
& = 1/2\min\{\widehat{\alpha}\widehat{M}^{-2},\widehat{\alpha}^{-3}\widehat{M}^4,\widehat{\alpha}^{-4}\widehat{M}^4\}\|z_h\|_{X_h}^2.
\end{align*}
The two inequalities above imply that
\begin{equation}
a_{\ep,h}(z_h,\mathbf{y}_h) \ge \beta\|z_h\|_{X_h}\|\mathbf{y}_h\|_{Y_h},
\end{equation}
where $\beta = 1/2\min\{\widehat{\alpha}\widehat{M}^{-2},\widehat{\alpha}^{-3}\widehat{M}^4,\widehat{\alpha}^{-4}\widehat{M}^4\}\min\{\widehat{\alpha}^{1/2},\widehat{\alpha}^2\widehat{M}^{-2}\}$. Therefore we have
\begin{equation}
\displaystyle\inf_{0\neq z_h\in X_{h}}\sup_{0\neq \mathbf{y}_h\in Y_h}\dfrac{a_{\ep,h}(z_h,\mathbf{y}_h)}{\|z_h\|_{X_h}\|\mathbf{y}_h\|_{Y_h}} \ge \beta
\end{equation}
The inf-sup condition follows.

\end{proof}

\section{Proof of Theorem \ref{Thm:ErrorEstimateinSpatialSemiDiscretization}}
\label{sec:6}
We need the following lemma, which is a direct consequence of Assumptions \ref{Assume:EqualityatExternalElement} and \ref{Assume:inequalities} (see \cite[Lemma 2.3]{tho06} for example).

\begin{lemma}
Letting $l$ be an integer with $2\le l\le k+1$, then there exist two positive constants $C_1$ and $C_2$ such that
\begin{subequations}
\begin{gather}
\|v-\Pi_h v\|_{V_h}\le C_1h^{l-1}\|v\|_{H^l(\Omega)}\label{Eq:VhError},\\
\|v-\Pi_hv\|_{V}\le C_2h^{l-1}\|v\|_{H^l(\Omega)}\label{Eq:VError}
\end{gather}
\end{subequations}
for all $v\in H^l(\Omega)$.
\end{lemma}

We recall that the exact solution is $u$ assumed to belong to $X_{l,m} = H^1(J;H^l(\Omega),H^m(\Omega))$ for $2\le l,m\le k+1$. Then $\Pi_h$ defines a projection from $X_{l,m}$ to $X_h$. We designate by $\Pi_h$ again, i.e., \begin{align*}
(\Pi_h u)(t) = \Pi_hu(t)\quad(t\in \overline{J}),\qquad \partial_t(\Pi_hu)(t) = \Pi_h(\partial_tu(t))\quad(t\in J).
\end{align*}
The estimate in Theorem \ref{Lem:estimate_from_consistency} is valid for $z_h = \Pi_hu$, which gives the following proof.

\begin{proof}[Proof of Theorem \ref{Thm:ErrorEstimateinSpatialSemiDiscretization}]
It is readily apparent that $X_{\ell,m}\subset X_V$. Therefore Theorem \ref{Lem:estimate_from_consistency} implies
\begin{align*}
\|u-u_{\ep,h}\|_{X_h}^2 &\le \displaystyle\inf_{z_h\in X_h}C\|u-z_h\|_{X_V}^2\\
&\le \|u-\Pi_hu\|_{X_V}^2\\
&=\displaystyle\int_J\left(\|u-\Pi_hu\|_{V}^2+\|\partial_t(u-\Pi_hu)\|_{V_h'}^2\right)dt+\|u(0)-\Pi_hu(0)\|_{L^2(\Omega)}^2.
\end{align*}
We can estimate that
\begin{equation}
\|u-\Pi_h u\|_{V}^2 \le Ch^{2(l-1)}\|u\|_{H^l(\Omega)},
\end{equation}
\begin{equation}
\|\partial_t(u-\Pi_h u)|_{V_h'}^2 \le \|\partial_tu-\Pi_h(\partial_tu)\|_{L^2(\Omega)}\le Ch^{2m}\|\partial_tu\|_{H^m(\Omega)}.
\end{equation}
Moreover, because
\begin{equation}
X_{l,m} = H^1(J;H^l(\Omega),H^m(\Omega))\subset H^1(J;H^j(\Omega),H^j(\Omega))\subset C^0(\overline{J};H^j(\Omega)),
\end{equation}
where $j=\min\{l,m\}$, we have $u(0)\in H^j(\Omega)$, and
\begin{equation}
\|u(0)-\Pi_hu(0)\|_{L^2(\Omega)}^2 \le Ch^{2j}\|u(0)\|_{H^j(\Omega)}^2.
\end{equation}
Summing up those estimates, we complete the proof.
\end{proof}

\section{Full discretization}
\label{sec:7}
Let $N\in\mathbb{N}$ be the number of time steps, $\tau=T/N$ and $t_n = n\tau$.
%\subsection{Finite difference method: Implicit Euler scheme}
We now consider the temporal discretization with implicit Euler (backward Euler) method. 

\smallskip

\noindent\textbf{(P$_{\ep,h,\tau}$)} 
Find $\{u_{\varepsilon,h,\tau}^n\}_{n=0}^N\in (V_h)^{N+1}$ such that
%\begin{equation}
%\label{Eq:FullDiscretizedProblem}
\begin{subequations}
\begin{gather}
\dfrac{1}{\tau}(u_{\varepsilon,h,\tau}^n-u_{\varepsilon,h,\tau}^{n-1},v_h)_{L^2(\Omega)} + a_{\ep,h}(t_n;u_{\varepsilon,h,\tau}^n,v_h) = F_{\ep,h}(t_n;v_h)\quad(v_h\in V_h),\label{Eq:FullDiscretizedProblem}\\
(u_{\varepsilon,h,\tau}^0-u_0,\widetilde{v}_h)_{L^2(\Omega)}=0\quad(\widetilde{v}_h\in V_h)\label{Eq:FullDiscreteInitioal}
\end{gather}
\end{subequations}

Clearly, $u_{\varepsilon,h,\tau}^0 = u_{\varepsilon,h}(0)$, where $u_{\varepsilon,h}\in X_h$ represents the solution of \textbf{(P$_{\ep,h}$)}.

\begin{lemma}
The problem \textup{\textbf{(P$_{\ep,h,\tau}$)}} has a unique solution $\{u_{\varepsilon,h,\tau}^n\}_{n=0}^N\in (V_h)^{N+1}$.
%There exists a unique solution $\{u_{\varepsilon,h,\tau}^n\}_{n=0}^N\in (V_h)^{N+1}$ of the equation (\ref{Eq:FullDiscretizedProblem}).
\end{lemma}

\begin{proof}
It is readily apparent that $u_{\varepsilon,h,\tau}^0\in V_h$. We take the basis functions $\{\phi_i\}_{i=1}^N$ of $V_h$, where $N=\operatorname{dim}V_h$, and
\begin{equation}
u_{\varepsilon,h,\tau}^n:=\displaystyle\sum_{i=1}^{N}u_{\varepsilon,h,\tau,i}^n\phi_i .
\end{equation}
Then, the equation \eqref{Eq:FullDiscretizedProblem} implies the following for $n=1,\cdots,N$,
\begin{equation}
\left(\mathbf{M}+\Delta t\mathbf{A}^n\right)\mathbf{u}_{\varepsilon,h,\tau}^n= \mathbf{M}\mathbf{u}_{\varepsilon,h,\tau}^{n-1}+\Delta t\mathbf{F}
^n,
\end{equation}
where $(\mathbf{M})_{ij}:=(\phi_i,\phi_j)_{L^2(\Omega)}$, $(\mathbf{A}^n)_{ij}:=a_{\ep,h}(t_n;\phi_i,\phi_j)$, $(\mathbf{F}^n)_i:=F_{\ep,h}(t_n;\phi_i)$ and $\mathbf{u}_{\varepsilon,h,\tau}^n:=(u_{\varepsilon,h,\tau,1}^n,\cdots,u_{\varepsilon,h,\tau,N})^{\operatorname{T}}$.  
Theorem \ref{Thm:Coercivity_of_b} gives $\mathbf{M}+\Delta t\mathbf{A}^n$ as a positive definite matrix.
Therefore, there exists a unique $u_{\varepsilon,h,\tau}^n$ for any $u_{\varepsilon,h,\tau}^{n-1}$.
\end{proof}

Let 
\begin{equation}
\mathcal{S}_{\tau} = \left\{w_h:[0,T]\to V_h \mid
\begin{array}{c}
\mbox{ For all }n=1,\cdots, N,\mbox{ there exists }v_h\in V_h\\
\mbox{ such that }w_h|_{(t_{n-1},t_n]}=v_h, \mbox{ and }w_h(0)\in V_h
\end{array}\right\},
\end{equation}
and $v_{h,\tau}(t_n^+) = \lim_{t\to t_n+0}v_{h,\tau}(t)$ for $v_{h,\tau}\in\mathcal{S}_{\tau}$.
It is noteworthy that $\mathcal{S}_{\tau}\subset L^2\left(J;V_h\right)$. The solution of \textbf{(P$_{\ep,h,\tau}$)} can be extended to an element of $\mathcal{S}_{\tau}$ as 
\begin{equation}
u_{\varepsilon,h,\tau}(t) = \left\{\begin{array}{ll}
u_0&\mbox{if }t=0,\\
u_{\varepsilon,h,\tau}^{n+1}&\mbox{if }t\in (t_{n},t_{n+1}]\quad(n=0,\cdots, N-1).
\end{array}\right.
\end{equation}

One can show that $u_{\varepsilon,h,\tau}\in\mathcal{S}_{\tau}$ satisfies the following estimate.

\begin{lemma}
\label{Lem:ErrorEstimateinTemporalDiscretization}
Assuming that $u_{\varepsilon,h}\in C^2\left(\overline{J};V_h\right)$, then let $\Pi_{\tau}u_{\varepsilon,h}\in\mathcal{S}_{\tau}$ be
\begin{equation}
\Pi_{\tau}u_{\varepsilon,h}(t) = \left\{\begin{array}{ll}
u_0&\mbox{if }t=0,\\
u_{\varepsilon,h}(t_{n+1})&\mbox{if }t\in(t_n,t_{n+1}] \quad(n=0,\cdots, N-1)
\end{array}\right.
\end{equation}
and $\rho = u_{\varepsilon,h,\tau}-\Pi_{\tau}u_{\varepsilon,h}\in\mathcal{S}_{\tau}$,
where $u_{\varepsilon,h}\in X_h$ is the solution of \textup{\textbf{(P$_{\ep,h}$)}}. Based on the relations presented above,  
\begin{subequations}
\begin{align}
&\displaystyle\max_{n=0,\cdots, N-1}\|\rho(t_n^+)\|_{L^2(\Omega)}\le \sqrt{\dfrac{T}{8\widehat{\alpha}}}\tau\|\partial_t^2u_{\varepsilon,h}\|_{L^{\infty}(J;L^2(\Omega))},\\
&\|\rho\|_{L^2(J;V_h)}\le\dfrac{\sqrt{T}}{2\widehat{\alpha}}\tau\|\partial_t^2u_{\varepsilon,h}\|_{L^{\infty}(J;L^2(\Omega))}.
\end{align}
\end{subequations}
\end{lemma}

\begin{proof}
Clearly, $\rho(t_0)=0$. In fact, $u_{\varepsilon,h}\in C^2\left(\overline{J};V_h\right)$ and the definition of $u_{\ep,h,\tau}$ yield
\begin{subequations}
\begin{align}
u_{\varepsilon,h,\tau}^{n+1} = u_{\varepsilon,h,\tau}(t_n^+) = u_{\varepsilon,h,\tau}(t_{n+1}), \label{Eq:FullDiscrete_Equal1}\\
u_{\varepsilon,h}(t_{n+1}) = \Pi_{\tau}u_{\varepsilon,h}(t_n^+) = \Pi_{\tau}u_{\varepsilon,h}(t_{n+1}) \label{Eq:FullDiscrete_Equal2}.
\end{align}
\end{subequations}
Using \eqref{Eq:FullDiscrete_Equal1}, equation \eqref{Eq:FullDiscretizedProblem} is written, equivalently, as
\begin{multline}
(u_{\varepsilon,h,\tau}(t_n^+)-u_{\varepsilon,h,\tau}(t_n),v_{h,\tau}(t_n^+))_{L^2(\Omega)} + \tau a_{\ep,h}(t_{n+1};u_{\varepsilon,h,\tau}(t_{n}^+),v_{h,\tau}(t_n^+))\\=\tau F_{\ep,h}(t_{n+1};v_{h,\tau}(t_n^+))\quad (v_{h,\tau}\in\mathcal{S}_{\tau}).
\end{multline}

%As the implicit Euler scheme we have
%\begin{align*}
%0 &= (u_{\varepsilon,h,\tau}^{n+1}-u_{\varepsilon,h,\tau}^n,v_{h,\tau}(t_n^+))_{L^2(\Omega)} + \tau a_{\ep,h}(t_{n+1};u_{\ep,h,\tau}^{n+1},v_{h,\tau}(t_n^+))-\tau F_{\ep,h}(t_{n+1};v_{h,\tau}(t_n^+))\\
%&= \left(u_{\varepsilon,h,\tau}(t_n^+)-u_{\varepsilon,h,\tau}(t_n),v_{h,\tau}(t_n^+)\right)_{L^2(\Omega)} + \tau a_{\ep,h}(t_{n+1};u_{\varepsilon,h,\tau}(t_{n}^+),v_{h,\tau}(t_n^+))-\tau F_{\ep,h}(t_{n+1};v_{h,\tau}(t_n^+))
%\end{align*}
%for all $v_{h,\tau}\in\mathcal{S}_{\tau}$ and $n=0,\cdots, N-1$. 

Furthermore, equations \eqref{Eq:FullDiscrete_Equal2} and \eqref{Eq:Nitsche_for_each_time} yield
%\begin{equation}
%(\partial_tu_{\varepsilon,h},v_h)_{L^2(\Omega)}+a_{\ep,h}(t;u_{\varepsilon,h},v_h) = F_{\ep,h}(t;v_h)\quad(v_h\in V_h,~ t\in \overline{J}).
%\end{equation}
%Therefore we have
%\begin{align*}
%0&= \tau(\partial_tu_{\varepsilon,h}(t_{n+1}),v_{h,\tau}(t_n^+))_{L^2(\Omega)} + \tau a_{\ep,h}(t_{n+1};u_{\varepsilon,h},v_{h,\tau}(t_n^+)) - \tau F_{\ep,h}(t_{n+1};v_{h,\tau}(t_n^+))\\
%&= (\Pi_{\tau}u_{\varepsilon,h}(t_n^+)-\Pi_{\tau}u_{\varepsilon,h}(t_n) + \tau^2/2\partial_t^2u_{\varepsilon,h}(c_n),v_{h,\tau}(t_n^+))_{L^2(\Omega)}\\
%&\hspace{3cm}+ \tau a_{\ep,h}(t_{n+1};\Pi_{\tau}u_{\varepsilon,h}(t_n^+),v_{h,\tau}(t_n^+)) - \tau F(t_{n+1};v_{h,\tau}(t_n^+))
%\end{align*}
%for all $v_{h,\tau}\in\mathcal{S}_{\tau}$, $t\in(t_n,t_{n+1}]$ 
\begin{multline}
(\Pi_{\tau}u_{\varepsilon,h}(t_n^+)-\Pi_{\tau}u_{\varepsilon,h}(t_n) + \tau^2/2\partial_t^2u_{\varepsilon,h}(c_n),v_{h,\tau}(t_n^+))_{L^2(\Omega)}+ \tau a_{\ep,h}(t_{n+1};\Pi_{\tau}u_{\varepsilon,h}(t_n^+),v_{h,\tau}(t_n^+))\\
 = \tau F_{\ep,h}(t_{n+1};v_{h,\tau}(t_n^+))\quad(v_{h,\tau}\in\mathcal{S}_{\tau})
\end{multline}
for some $c_n\in (t_n,t_{n+1}]$. 
By the two equations presented above, setting $v_{h,\tau} = \rho\in \mathcal{S}_{\tau}$ leads to
\begin{equation}
(\rho(t_n^+)-\rho(t_n),\rho(t_n^+))_{L^2(\Omega)} - \tau^2/2(\partial_t^2u_{\varepsilon,h}(c_n),\rho(t_n^+))_{L^2(\Omega)} + \tau a_{\ep,h}(t_{n+1};\rho,\rho) =0.
\end{equation}
Now we can estimate
\begin{subequations}
\begin{align*}
(\rho(t_n^+)-\rho(t_n),\rho(t_n^+))_{L^2(\Omega)} &= \dfrac{1}{2}(\|\rho(t_n^+)-\rho(t_n)\|_{L^2(\Omega)}^2 + \|\rho(t_n^+)\|_{L^2(\Omega)}^2 - \|\rho(t_n)\|_{L^2(\Omega)}^2)\\
&\ge \dfrac{1}{2}(\|\rho(t_n^+)\|_{L^2(\Omega)}^2 - \|\rho(t_n)\|_{L^2(\Omega)}^2),
\end{align*}
\begin{align*}
- \tau^2/2(\partial_t^2u_{\varepsilon,h}(c_n),\rho(t_n^+))_{L^2(\Omega)} &\ge -\tau^{3/2}/2\|\partial_t^2u_{\varepsilon,h}\|_{L^{\infty}(J;L^2(\Omega))}\|\rho\|_{L^2(t_n,t_{n+1};V_h)}\\
&\ge -\left(\dfrac{\tau^{3}}{16q}\|\partial_t^2u_{\varepsilon,h}\|_{L^{\infty}(J;L^2(\Omega))}^2+q\|\rho\|_{L^2(t_n,t_{n+1};V_h)}^2\right) 
\end{align*}
for any $q\in\mathbb{R}$, and
\begin{equation*}
\tau a_{\ep,h}(t_{n+1};\rho,\rho) \ge \widehat{\alpha}\|\rho\|_{L^2(t_n,t_{n+1};V_h)}.
\end{equation*}
\end{subequations}

%\begin{align*}
%0&=(\rho(t_n^+)-\rho(t_n),\rho(t_n^+))_{L^2(\Omega)} + \tau a_{\ep,h}(t_{n+1};\rho,\rho) - \tau^2/2(\partial_t^2u_{\varepsilon,h}(c_n),\rho(t_n^+))_{L^2(\Omega)}\\
%&\ge \dfrac{1}{2}(\|\rho(t_n^+)-\rho(t_n)\|_{L^2(\Omega)}^2 + \|\rho(t_n^+)\|_{L^2(\Omega)}^2 - \|\rho(t_n)\|_{L^2(\Omega)}^2) + \widehat{\alpha}\displaystyle\int_{t_n}^{t_{n+1}}\|\rho\|_{V_h}^2~dt\\
%&\qquad - \dfrac{\tau^{3/2}}{2}\|\partial_t^2u_{\varepsilon,h}\|_{L^{\infty}(J;L^2(\Omega))}\|\rho\|_{L^2(t_n,t_{n+1};V_h)}\\
%&\ge \dfrac{1}{2}(\|\rho(t_n^+)\|_{L^2(\Omega)}^2 - \|\rho(t_n)\|_{L^2(\Omega)}^2)+\widehat{\alpha}\|\rho\|_{L^2(t_n,t_{n+1};V_h)}^2\\
%&\qquad-(\dfrac{\tau^{3}}{16q}\|\partial_t^2u_{\varepsilon,h}\|_{L^{\infty}(J;L^2(\Omega))}^2+q\|\rho\|_{L^2(t_n,t_{n+1};V_h)}^2)
%\end{align*}
%for any $q\in\mathbb{R}$. 
Therefore, letting $q = \widehat{\alpha}$ yields
\begin{align*}
\|\rho(t_n^+)\|_{L^2(\Omega)}^2&\le \|\rho(t_{n-1}^+)\|_{L^2(\Omega)}^2+\dfrac{\tau^{3}}{8\widehat{\alpha}}\|\partial_t^2u_{\varepsilon,h}\|_{L^{\infty}(J;L^2(\Omega))}^2\\
&\le \|\rho(t_0)\|_{L^2(\Omega)}^2 + \dfrac{\tau^{3}(n+1)}{8\widehat{\alpha}}\|\partial_t^2u_{\varepsilon,h}\|_{L^{\infty}(J;L^2(\Omega))}^2\\
&\le \dfrac{T}{8\widehat{\alpha}}\tau^2\|\partial_t^2u_{\varepsilon,h}\|_{L^{\infty}(J;L^2(\Omega))}^2
\end{align*}
for all $n=0,\cdots, N=1$. Furthermore, if we set $q=\dfrac{\widehat{\alpha}}{2}$, then we have
\begin{equation}
\widehat{\alpha}\|\rho\|_{L^2(t_n,t_{n+1};V_h)}^2+\|\rho(t_n^+)\|_{L^2(\Omega)}^2 - \|\rho(t_n)\|_{L^2(\Omega)}^2\le \dfrac{\tau^3}{4\widehat{\alpha}}\|\partial_t^2u_{\varepsilon,h}\|_{L^{\infty}(J;L^2(\Omega))}^2.
\end{equation}
Summing this from $n=0$ to $n=N-1$ gives
\begin{align*}
\|\rho\|_{L^2(J;V_h)}^2 &\le \dfrac{1}{\widehat{\alpha}}\|\rho(t_0)\|_{L^2(\Omega)}^2 + \dfrac{\tau^3N}{4\widehat{\alpha}^2}\|\partial_t^2u_{\varepsilon,h}\|_{L^{\infty}(J;L^2(\Omega))}^2\\
&\le \dfrac{T}{4\widehat{\alpha}^2}\tau^2\|\partial_t^2u_{\varepsilon,h}\|_{L^{\infty}(J;L^2(\Omega))}^2.
\end{align*}
\end{proof}

\begin{theorem}[Error estimate]
\label{Thm:Error_for_Full_Discrete}
Let $l,m$ be two integers satisfying $2\le l,m\le k+1$. Assuming that $u\in X_{l,m}$ and $u_{\varepsilon,h}\in C^2(\overline{J};V_h)$, then there exists a positive constant $C$ such that
\begin{align*}
\|u-u_{\varepsilon,h,\tau}\|_{L^2(J;H^1(\Omega))}^2 &\le C\Biggl(\displaystyle\int_J(h^{2(l-1)}\|u\|_{H^{l}(\Omega)}^2+h^{2m}\|\partial_tu\|_{H^m(\Omega)}^2)~dt+h^{2j}\|u(0)\|_{H^{j}(\Omega)}^2\\
&\qquad + \tau^2\|\partial_tu_{\varepsilon,h}\|_{L^2(J;H^1(\Omega))}^2 + \dfrac{T}{4\widehat{\alpha}^2}\tau^2\|\partial_t^2u_{\varepsilon,h}\|_{L^{\infty}(J;L^2(\Omega))}^2\Biggr),
\end{align*}
where $j:=\min\{l,m\}$.
\end{theorem}

\begin{proof}
We have
\begin{equation*}
\|u-u_{\varepsilon,h,\tau}\|_{L^2(J;H^1(\Omega))}^2 \le \|u-u_{\varepsilon,h}\|_{X_h}^2 + \|u_{\varepsilon,h}-\Pi_{\tau}u_{\varepsilon,h}\|_{L^2(J;H^1(\Omega))}^2 + \|\Pi_{\tau}u_{\varepsilon,h}-u_{\varepsilon,h,\tau}\|_{L^2(J;V_h)}^2.
\end{equation*}
Here Theorem \ref{Thm:ErrorEstimateinSpatialSemiDiscretization}, Lemma \ref{Lem:ErrorEstimateinTemporalDiscretization} and the approximation error estimate for piecewise constant yield the result.
\end{proof}

\section{Numerical examples}
\label{sec:8}

\begin{figure}[tb]
\centering
\includegraphics[width=\linewidth]{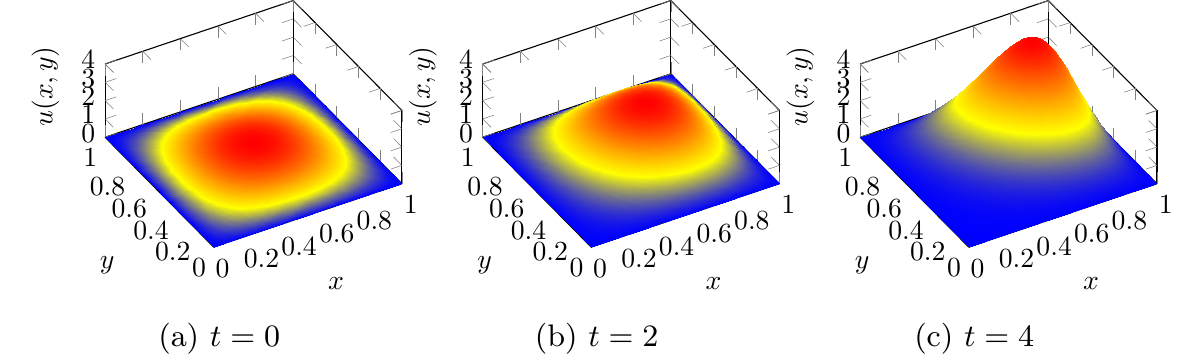}
\caption{Exact solutions at different time steps.}
\label{Fig:ExactSolutionforNitsche}
\end{figure}

\begin{figure}[tb]
\centering
\includegraphics[width=\linewidth]{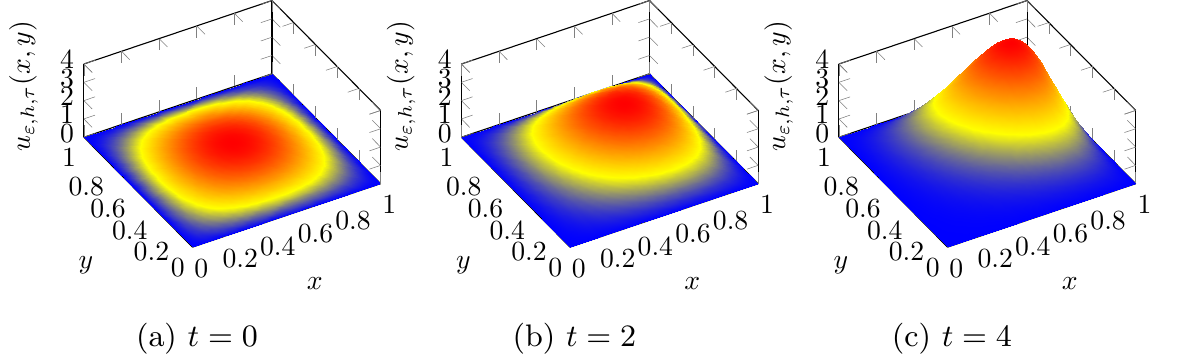}
\caption{Numerical solutions of the Nitsche method at different time steps by consideration of a uniform mesh with mesh size $h=1/30$.}
\label{Fig:NumericalSolutionforNitsche}
\end{figure}

\begin{figure}[tb]
\centering
\includegraphics[width=\linewidth]{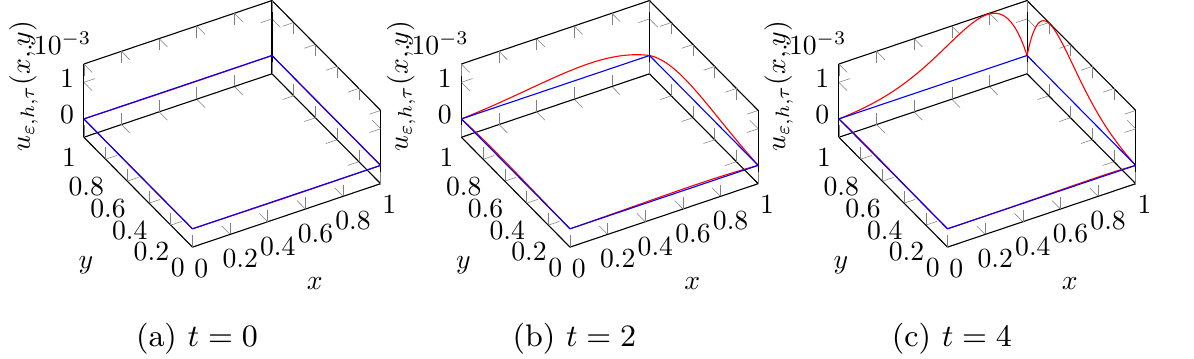}
\caption{Numerical solutions of the Nitsche method on boundary $\Gamma$, at different time steps by consideration of a uniform mesh with mesh size $h=1/30$.}
\label{Fig:OnBoundary}
\end{figure}

Our example is given as $\Omega := (0,1)^2$, $J=(0,4)$, and 
\begin{equation}
\left\{\begin{array}{rll}
\partial_tu - \Delta u + (1,1)^{\mathrm{T}}\cdot\nabla u + u &= f(\mathbf{x},t)&\mbox{in }\Omega\times J,\\
u &= 0&\mbox{on }\Gamma\times J,\\
u(\mathbf{x},0) &= \sin(\pi x)\sin(\pi y)&\mbox{for }\mathbf{x} = (x,y)\in \Omega.
\end{array}\right.
\end{equation}
That is, we let $\mu(\mathbf{x},t)=I$, $\mathbf{a}(\mathbf{x},t) =
(1,1)^{\mathrm{T}}$ and $c(\mathbf{x},t) = 1$.
One can readily check that this problem has a unique solution for any $f\in L^2(J;H^{-1}(\Omega))$. We let
\begin{align*}
f(x,y,t) &:= \left((x+y+2t-2t^2+2\pi^2)\sin(\pi x)\sin(\pi y)\right.\\
&\hspace{1cm}\left. + (\pi-2\pi t)\cos(\pi x)\sin(\pi y) + (\pi-2\pi t)\sin(\pi x)\cos(\pi y)\right)e^{(x+y-1)t},
\end{align*}
Then $f\in L^2(J;H^1(\Omega))$, and
\begin{equation}
u(x,y,t) :=  \sin(\pi x)\sin(\pi y)e^{(x+y-1)t}
\end{equation}
is the unique solution.
In Figure \ref{Fig:ExactSolutionforNitsche}, we show the exact solution at different time steps.

We use the $(k,k)$-th degree B-spline basis functions for spatial discretization using the uniform mesh and implicit Euler scheme for temporal discretization, where $k=1,2$: we consider the approximate problem \eqref{Eq:FullDiscretizedProblem}. We let $\tau = O(h^k)$, where $h$ is the mesh size for uniform mesh. Then we know from Theorem \ref{Thm:Error_for_Full_Discrete} that
\begin{equation}
\|u-u_{\varepsilon,h,\tau}\|_{L^2\left(J;H^1(\Omega)\right)} \le Ch^k.
\end{equation}

As shown in Figure \ref{Fig:NumericalSolutionforNitsche}, this report describes the numerical solution shape. Then we show the boundary value of the numerical solution in Figure \ref{Fig:OnBoundary}. The weak imposition of the Dirichlet boundary condition is actually observed because the boundary value does not vanish.

\begin{figure}[tb]
\centering
\includegraphics{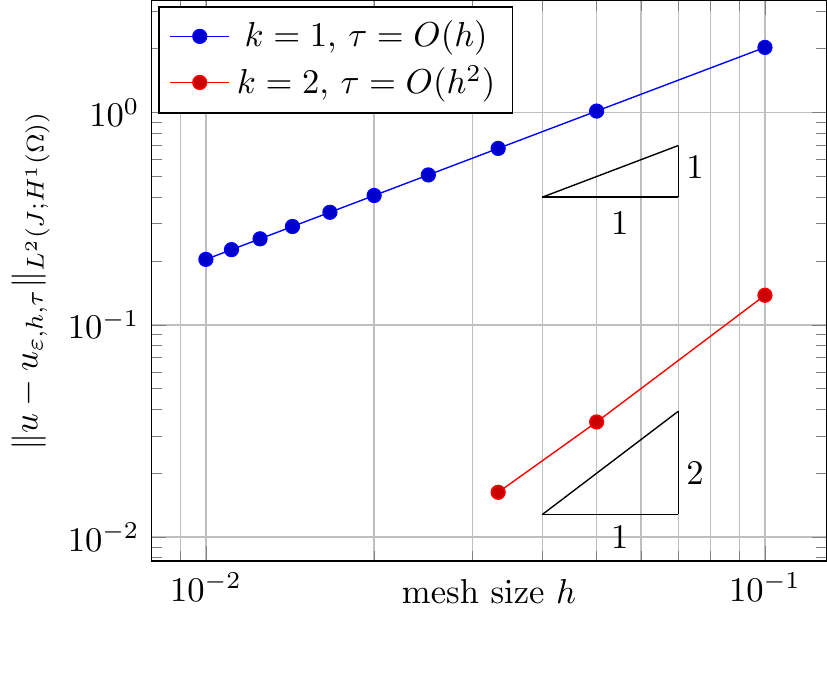}
\caption{ $L^2(J;H^1(\Omega))$ error on a uniform mesh.}
\label{Fig:ErrorPlotforNitsche}
\end{figure}

Furthermore, this report describes the error for uniform mesh in Figure \ref{Fig:ErrorPlotforNitsche}, which shows that the rate of convergence is approximately equal to $k$, which is expected by the theory.

%%%%%%%%%%%%%%%%%%%%%%%%%%%%%%%%%%%%%%%%%
%%%%%%%%%%%%%%%%%%%%%%%%%%%%%%%%%%%%%%%%%
\paragraph{Acknowledgements.}
This study was supported by JST CREST Grant Number JPMJCR15D1 and JSPS KAKENHI Grant Number 15H03635. 
The first author was also supported by the Program for Leading Graduate Schools, MEXT, Japan.

%%%%%%%%%%%%%%%%%%%%%%%%%%%%%%%%%%%%%%%%%
%%%%%%%%%%%%%%%%%%%%%%%%%%%%%%%%%%%%%%%%%
\bibliographystyle{plain}
%\bibliography{reference}
%\input{us18.bbl}

%%%
%%%
\end{document}